\documentclass[12pt]{amsart}

\usepackage{amssymb,amsmath, amsthm, amsfonts,amscd,latexsym, dsfont, color, xcolor}
\usepackage[hidelinks]{hyperref} 
\usepackage{graphics}
\usepackage{wrapfig}
\usepackage{pst-node}
\usepackage{tikz-cd} 
\usepackage{enumitem}
\usepackage{mathrsfs}
\newcommand{\bbR}{\mathbb R}
\newcommand{\RP}{\mathbb R \mathbb P}

\newcommand{\bbH}{\mathbb H}
\newcommand{\bbP}{\mathbb P}

\newcommand{\lamdba}{\lambda} 

    \usepackage{svg}
\usepackage{cleveref}

\DeclareFontFamily{U}{mathx}{\hyphenchar\font45}
\DeclareFontShape{U}{mathx}{m}{n}{<-> mathx10}{}
\DeclareSymbolFont{mathx}{U}{mathx}{m}{n}
\DeclareMathAccent{\widebar}{0}{mathx}{"73}

\crefformat{section}{\S#2#1#3} 
\crefformat{subsection}{\S#2#1#3}
\crefformat{subsubsection}{\S#2#1#3}
	\allowdisplaybreaks

\theoremstyle{definition} 

\usepackage{thmtools, thm-restate}

\theoremstyle{plain}

\newtheorem{theorem}{Theorem}[section]
\newtheorem{proposition}[theorem]{Proposition}
\newtheorem{lemma}[theorem]{Lemma}
\newtheorem{corollary}[theorem]{Corollary}

\newtheorem{definition}[theorem]{Definition}
\newtheorem{fact}[theorem]{Fact}

\newtheorem*{remark}{Remark}
\newtheorem{claim}[theorem]{Claim}
\newtheorem*{remarks}{Remarks}
\numberwithin{equation}{section}

\newtheorem*{theoremA}{Theorem A.}
\newtheorem*{theoremB}{Theorem B}
\newtheorem*{theoremC}{Theorem C.}
\newtheorem*{theoremD}{Theorem D}
 
	\usepackage[margin=1in]{geometry}
 \title[Metrics on Hitchin Components]{\vspace{-.2cm}Metrics on Hitchin Components from H\"older Distortion}
\author[A. Nolte]{Alexander Nolte}
\date{}

\begin{document}
\,
\maketitle
\allowdisplaybreaks
\vspace{-.35cm}
\begin{abstract}
    We observe Thurston's asymmetric metric on Teichm\"uller space may be expressed in terms of the H\"older regularity of boundary maps.
    We then associate $2$-dimensional stratified loci in $\RP^{n-1}$ to $\text{PSL}_n(\mathbb{R})$ Hitchin representations with $n > 3$.
    We prove that measuring the relative H\"older distortion of these loci gives asymmetric metrics on the Hitchin component $\text{Hit}_n(S)$ with complete and geometrically meaningful symmetrizations.
    These are the first known geometrically significant complete metrics on $\text{Hit}_n(S)$ for $n > 3$.
\end{abstract}
\stepcounter{section}

Let $\rho_a$ and $\rho_b$ be discrete and faithful (\textit{Fuchsian}) representations in $\text{PSL}_2(\bbR)$ of the fundamental group $\Gamma$ of a closed oriented hyperbolic surface $S$. There is a unique homeomorphism $\phi_{ab}$ of the boundary $\partial \bbH$ of the hyperbolic plane $\bbH$ that is equivariant with respect to the actions of $\rho_b$ and $\rho_a$.
Let $\alpha_{ab}$ be the supremum of the $\alpha$ for which $\phi_{ab}$ is $\alpha$-H\"older.

Observe $d(\rho_a,\rho_b) = -\log \alpha_{ab}$ is non-negative and satisfies the triangle inequality as supremal H\"older exponents are super-multiplicative under composition and valued in $[0,1]$. 

In fact, $d$ has been seen before: it is a new expression of Thurston's asymmetric metric $d_{\text{Th}}$ \cite{thurston1998minimal} on Teichm\"uller space $\mathcal{T}(S)$ (Theorem A). That is, $d(\rho_a, \rho_b) = \log \text{L}_{ab}$, where $\text{L}_{ab}$ is the infimal Lipschitz constant of identity-isotopic homeomorphisms from $\bbH/\rho_a(\Gamma)$ to $\bbH/\rho_b(\Gamma)$.

This paper centers on extensions of this idea to Hitchin components $\text{Hit}_n(S)$. These are the connected components of the $\text{PSL}_n(\bbR)$ character variety of $\Gamma$ that contain discrete and faithful representations of $\Gamma$ in $\text{PSL}_2(\bbR)$-subgroups of $\text{PSL}_n(\bbR)$ that act irreducibly on $\bbR^n$. The so-called \textit{Hitchin representations} that make up $\text{Hit}_n(S)$ demonstrate both remarkably similar behavior to Fuchsian representations in $\text{PSL}_2(\bbR)$ and rich ``higher rank'' phenomena. They have seen intense interest and study in recent years (e.g. \cite{bridgeman2015pressure} \cite{dancigerGK2024convex} \cite{guichard2008convex} \cite{guichard2012anosov} \cite{labourie2017cyclic} \cite{potrie2017eigenvalues-entropy} \cite{wienhard2018Invitation}).

\medskip 

Our main results are as follows. We associate to any $\text{PSL}_n(\bbR)$-Hitchin representation with $n > 3$ a particularly well-behaved locus $\mathcal{B}_\rho$ (a \textit{torus bouquet}) in $\RP^{n-1}$. We prove that measuring the H\"older regularity of canonical maps between these torus bouquets gives asymmetric metrics $D_B$ on $\text{Hit}_n(S)$ with completeness properties (Theorems B--D).
Notably, the symmetrizations $D_B^*$ associated to measuring bi-H\"older regularity are complete.
These are the first complete metrics on $\text{Hit}_n(S)$ for $n > 3$ that are defined by a geometric construction.

\medskip 

Much effort has been invested constructing and studying metrics on Hitchin components.
Two prior constructions are valid for all $n$: pressure metrics from dynamics extending the Weil-Petersson metric \cite{bridgeman2015pressure}, and a dynamical extension \cite{carvajales2022thurstonsasymmetric} of Thurston's asymmetric metric.

A basic feature of the theory is that the natural notion of lengths of curves associated to $\text{PSL}_n(\bbR)$ Hitchin representations is in general not valued in $\bbR^+$, but instead in the Weyl chamber $\mathfrak{a}^+$ of $\mathfrak{sl}_n(\bbR)$ (see \S \ref{sss-lie-notations}).
As a result, both prior constructions require a choice of scalar length measurement.
These constructions also rely on an entropy invariant that remains mysterious except in isolated cases \cite{potrie2017eigenvalues-entropy}.
In contrast, our construction requires neither a choice of length measurement nor the use of dynamical invariants such as entropy.

\subsection{Main Results} Let $\rho_a$ and $\rho_b$ now be Hitchin representations $\Gamma \to \text{PSL}_n(\bbR)$.
Let $\mathcal{F}(\bbR^n)$ be the manifold of full flags of subspaces of $\bbR^n$, i.e. of ($n-1$)-ples $(V_1, ..., V_{n-1})$ of nested subspaces of $\bbR^n$ of dimension $1, ..., n-1$.
Let $\partial \Gamma$ be the Gromov boundary of $\Gamma$, which is homeomorphic to the circle and is acted on by $\Gamma$.
As for Fuchsian representations, there are canonical equivariant maps $\xi_a, \xi_b : \partial \Gamma \to \mathcal{F}(\bbR^n)$ that respect the dynamics of the action of $\Gamma$ on $\partial \Gamma$ \cite{guichard2012anosov} \cite{labourie2004anosov}.
For $c \in \{a,b\}$ write $\xi_c = (\xi^1_c, ..., \xi^{n-1}_c)$, expanding out the flags.

\begin{definition}\label{def-k-couple-map}
    For $k =1,..., n-1$ the {\rm $k$-coupling map} $\phi_{ab}^k: \xi^k_b(\partial \Gamma) \to \xi^k_a(\partial \Gamma)$ is $\xi^k_a \circ (\xi^k_b)^{-1}$.
    Let $\alpha_{ab}^k$ be the supremum of the $\alpha > 0$ so $\phi_{ab}^k$ is $\alpha$-H\"older.
    
    The {\rm $k$-coupling distance $d_k$} on ${\rm{Hit}}_n(S)$ is defined by $d_k(\rho_a, \rho_b) = - \log \alpha_{ab}^k$. 
\end{definition}

Note that the domain and target of $\phi_{ab}^k$ are ``flipped'' relative to the arguments of $d_k$.

In \cite{carvajales2022thurstonsasymmetric}, Carvajales-Dai-Pozzetti-Wienhard extend Thurston's asymmetric metric to \textit{stretch metrics} $d_{\text{Th}}^{\varphi}$ on $\text{Hit}_n(S)$. Here $\varphi$ is a functional on $\mathfrak{a}^+$ that determines a length function $\ell^{\alpha_k}$ on $\text{PSL}_n(\bbR)$ (\S \ref{sss-lie-notations}).
We encounter the case where $\varphi$ is a simple root $\alpha_k$ ($k = 1, ..., n-1$), where $d_{\text{Th}}^{\alpha_k}(\rho_a, \rho_b) = \log  \, \sup_{\gamma \in\Gamma - \{e\}} [\ell^{\alpha_k}(\rho_b(\gamma))/\ell^{\alpha_k}(\rho_a(\gamma))]$ (\cite{carvajales2022thurstonsasymmetric} \cite{potrie2017eigenvalues-entropy}).
Our first theorem is: 

\begin{theoremA}[Coupling is Stretch]\label{thm-computation}
    On $\mathcal{T}(S)$, Thurston's asymmetric metric is the coupling metric: $d_1 = d_{{\rm{Th}}} $. On ${\rm{Hit}}_n(S)$, simple root stretch metrics are $k$-coupling metrics: $d_k = d_{{\rm{Th}}}^{\alpha_k}$.
\end{theoremA}

As an application of Theorem A, we obtain a simplified proof of a known comparison for $d_{\text{Th}}$.
Let $d_{\text{T}}$ be Teichm\"uller's metric on $\mathcal{T}(S)$, i.e. $2d_{\text{T}}(\rho_a, \rho_b)$ is the logarithm of the infimal quasiconformal constant of an identity-isotopic homeomorphism from $\bbH/\rho_a(\Gamma)$ to $ \bbH/\rho_b(\gamma)$.

\begin{corollary}[\cite{sorvali1973dilatation}, \cite{wolpert1979length}] $d_{\rm{Th}} \leq 2 d_{\rm{T}}$.
    
\end{corollary}

\begin{proof}
    Combine Theorem A, Ahlfors' classical lower bound of $K^{-1}$ for the H\"older regularity of a $K$-quasiconformal disk homeomorphism \cite{ahlfors1954quasiconformal}, and the definitions of $d_{\text{T}}$ and $d_{\text{Th}}$.
\end{proof}

Theorem A is a corollary of Theorem 1.9 in \cite{tsouvalas2023holder}. The reformulations of $d_\text{Th}$ and $d_{\text{Th}}^{\alpha_k}$ have not been observed before.
We give a different proof than \cite{tsouvalas2023holder} of the relevant regularity estimate for $\mathcal{T}(S)$ in \S \ref{s-teich-space-case}.
Both \cite{tsouvalas2023holder} and our proof rely crucially on \cite{abelsMargulisSoifer1995semigroups}; our technique is new.
The proof we present in this case is written to not use higher-rank machinery.

\subsubsection{Hitchin Components} Our main result concerns metrics on Hitchin components that similarly measure H\"older regularity.
In \S \ref{ss-bouquet-couplings}, for $\rho \in \text{Hit}_n(S)$ with $n > 3$, we construct a stratified $2$-dimensional set $\mathcal{B}_\rho$ in $\RP^{n-1}$.
We call $\mathcal{B}_\rho$ the \textit{torus bouquet} of $\rho$.
It is a finite union of tori that are $C^1$ away from a distinguished curve. See Figure \ref{fig-fuchsian-in-chart} (Left) for an example.

As before, for any $\rho_a, \rho_b \in \text{Hit}_n(S)$ there is a canonical homeomorphism $\Phi_{ab} : \mathcal{B}_{\rho_b} \to \mathcal{B}_{\rho_a}$ conjugating $\rho_b$ to $\rho_a$.
Paralleling Definition \ref{def-k-couple-map} above, we define a \textit{bouquet coupling metric} $D_B$ on $\text{Hit}_n(S)$ (Definitions \ref{def-bouquet-map}-\ref{def-bouquet-dist}).
We call the symmetrization $D_B^*$ of $D_B$ associated to measuring bi-H\"older regularity the \textit{bouquet bi-coupling metric}.
Both $D_B$ and $D_B^*$ are mapping class group invariant.
Our main theorem follows from bounds on $D_B$ in Theorem C below:

\begin{theoremB}
    For $n > 3$, the bouquet bi-coupling metric $D_B^*$ is a complete metric on ${\rm{Hit}}_n(S)$.
\end{theoremB}

The metric $D_B^*$ is the first geometrically significant complete metric on ${\rm{Hit}}_n(S)$ for $n > 3$.
Pressure metrics are never complete. The metrics of \cite{carvajales2022thurstonsasymmetric} are expected to usually be incomplete, and Dai has shown that for $n = 3$ only the simple root metrics $d_{\text{Th}}^{\alpha_k}$ have complete symmetrizations \cite{dai2024email}.
Bouquet (bi-)coupling metrics have the further appeal of being ``purely geometric,'' in that they are defined---without entropy---in terms of geometric constructions.

\subsection{Bouquet Coupling}\label{ss-bouquet-couplings} We define the torus bouquets and associated metrics of our main construction here.
Our construction captures enough information to produce metrics, while being clean and independent of any apparent choices.
What seems essential to our results is that torus bouquets are \textit{sufficiently rich} and \textit{canonically parametrized} in terms of limit data.

We remark that implementing our perspective with the limit curve $\xi(\partial \Gamma) \subset \mathcal{F}(\bbR^n)$ seems much less tractable than with torus bouquets.
In particular, it is not clear that the objects obtained for $n > 3$ even separate points. See \S \ref{sss-full-flag-bad}.

\subsubsection{The Torus Bouquet} We begin with definitions. Let $\rho : \Gamma \to \text{PSL}_n(\bbR)$ for $n > 3$ be Hitchin with limit map $\xi = (\xi^1, ..., \xi^{n-1})$, let $\mathcal{K}_n = \{2, ..., \lfloor n/2 \rfloor\}$, and let $k_n = \lfloor n/2 \rfloor$ where $\lfloor \cdot \rfloor$ is the floor function.
For $k \in \mathcal{K}_n$, define $\Phi_\rho^k: \partial \Gamma^2 \to \RP^{n-1}$ by
\begin{align}
    \Phi_\rho^k (x,y) &= \begin{cases}
        \xi^k(x) \cap \xi^{n+1-k}(y) & x \neq y, \\
        \xi^1(x) & x = y,
    \end{cases}
\end{align} and define $T_\rho^k$ to be $\Phi_\rho^k(\partial \Gamma^2)$.
For $n = 4$ the torus $T_\rho^2$ is the boundary of the domain of discontinuity for $\rho$ in $\RP^3$ \cite{guichard2008convex}, and is rendered in Figure 1, Left.

We next define the \textit{$n$-bouquet} $B_n$ to be a gluing of $k_n -1$ copies $\partial \Gamma^{2}_l$ of $\partial \Gamma^{2}$ ($l = 2, ..., k_n)$ along their diagonals, i.e. by the equivalence relation $\sim$ specified by
\begin{align}
    p \sim q , \qquad p = (x,x) \in \partial \Gamma^{2}_{l}, \, q = (x,x) \in \partial \Gamma^{2}_m \, \text{ where } x\in \partial \Gamma \text{ and } l, m \in \mathcal{K}_n.
\end{align}
The bouquet $B_n$ carries a natural $\Gamma$-action.

Let $\Phi_\rho : B_n \to \RP^{n-1}$ be the map induced by $\{\Phi_\rho^k\}_{k=2,..., k_n}$ and define $\mathcal{B}_\rho = \Phi_\rho(B_n)$. Then $\Phi_\rho$ is a homeomorphism onto $\mathcal{B}_\rho$ (\S \ref{s-bouquet-coupling}). 
We call $\mathcal{B}_\rho$ the \textit{torus bouquet} of $\rho$.
See Figure \ref{fig-fuchsian-in-chart}.

\begin{figure}
    \centering
    \includegraphics[scale=0.53]{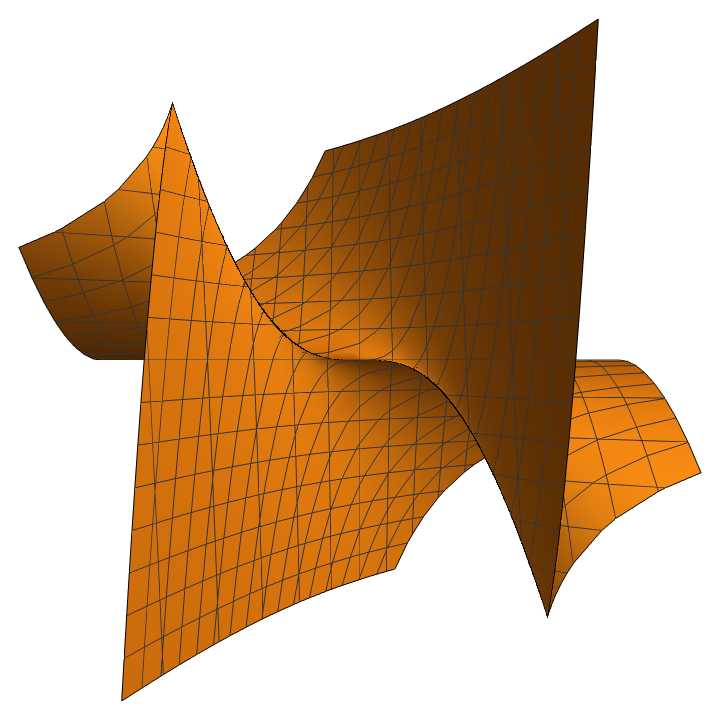} \hspace{1cm} \includegraphics[scale=0.68]{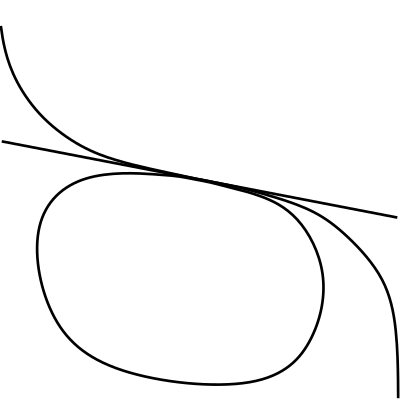}
    \caption{Left: the part of $\mathcal{B}_\rho$ for a $\text{PSL}_4(\bbR)$-Hitchin representation in a coordinate box in an affine chart for $\RP^3$. This is the ``Fuchsian'' example, i.e. is preserved by an irreducible $\text{PSL}_2(\bbR)$-subgroup of $\text{PSL}_4(\bbR)$. The singular curve is the image of $\xi^1_\rho$. Right: sketch of the intersection of $\mathcal{B}_\rho$ with a $3$-dimensional subspace of the form $\xi^4_\rho(x)$ for a $\text{PSL}_8(\bbR)$-Hitchin representation $\rho$.
    The point where the drawn curves intersect is $\xi^1_\rho(x)$.
    The term ``bouquet'' is motivated by how the tori $T_\rho^k$ come together near the image of $\xi^1_\rho.$}
    \label{fig-fuchsian-in-chart}
\end{figure}

\subsubsection{Bouquet Coupling Metrics} We now describe the associated construction of metrics.

\begin{definition}\label{def-bouquet-map}
For ${\rm{PSL}}_n(\bbR)$ Hitchin representations $\rho_a$ and $\rho_b$ $(n > 3)$, the {\rm{bouquet coupling map}} $\Phi_{ab}: \mathcal{B}_{\rho_b} \to \mathcal{B}_{\rho_a}$ is $\Phi_{\rho_a} \circ \Phi_{\rho_b}^{-1}$.
Let $\alpha_{ab}$ be the supremum of the $\alpha > 0$ so that $\Phi_{ab}$ is $\alpha$-H\"older and let $\beta_{ab}$ be the supremum of the $\alpha > 0$ so that $\Phi_{ab}$ is $\alpha$-bi-H\"older.
\end{definition}

\begin{definition}\label{def-bouquet-dist}
    The {\rm{bouquet coupling distance}} $D_B$ on ${\rm{Hit}}_n(S)$ ($n > 3$) is $D_B(\rho_a, \rho_b) = -\log \alpha_{ab}$.
    The {\rm{bouquet bi-coupling distance}} $D_B^*$ on ${\rm{Hit}}_n(S)$ is $D_B^* (\rho_a, \rho_b) = - \log \beta_{ab}$.
\end{definition}

This gives asymmetric and symmetric metrics that we bound in terms of eigenvalue data.
Notationally, for $\varphi \in \mathfrak{a}^*$ with $\mathfrak{a}$ a Cartan subalgebra of $\mathfrak{sl}_n(\bbR)$ the $\varphi$\textit{-length} $\ell^{\varphi}_\rho(\gamma)$ is $\varphi$ evaluated on the Jordan projection $\lambda(\rho(\gamma))$ of $\rho(\gamma)$ to the Weyl chamber $\mathfrak{a}^+$. See \S \ref{sss-lie-notations}.
We write $\alpha_i$ as the $i$-th simple root $(i = 1, ..., n-1)$ and $\text{H}$ as the Hilbert functional $\alpha_1 + \cdots + \alpha_{n-1}$.

\begin{theoremC}[Metricity, Bounds]\label{thm-metric-bounds}
    For $n > 3$ the bouquet coupling metric $D_B$ is an asymmetric metric on ${\rm{Hit}}_n(S)$, and $D_B^*$ is a metric on ${\rm{Hit}}_n(S)$. For $\rho_a, \rho_b \in {\rm{Hit}}_n(S)$, 
    \begin{align}
        \max_{k=1,...,n-1}d_{\rm{Th}}^{\alpha_k}(\rho_a, \rho_b) \leq D_B(\rho_a, \rho_b) \leq \max \left( \max_{k=1,...,n-1}d_{\rm{Th}}^{\alpha_k}(\rho_a, \rho_b), \, \log \sup_{\gamma \in \Gamma - \{e\}} \frac{\ell_{\rho_b}^{\rm{H}}(\gamma)}{\ell^{\alpha_1}_{\rho_a}(\gamma) } \right).
    \end{align}
\end{theoremC}

The interest in and import of a metric on a moduli space is closely correlated to the richness of the phenomena it measures.
Torus bouquets reflect interesting aspects of the geometry of Hitchin representations, while being fairly tractable due to their low-dimensionality.

More specifically, in \S \ref{s-bouquet-coupling} we show each torus $T_\rho^k$ has two canonical foliations by hyperconvex Frenet curves in projective spaces (see \S\ref{sss-hyperconvexity}).
The leaves of these foliations are curious objects.
In the case of $n =4$, one foliation is by the boundaries of leaves of the properly convex foliated projective structure associated to $\rho$ \cite{guichard2008convex}.
These domains exhibit have projective-geometric properties.
For instance, let $\mathfrak{C}$ be collection of projective classes of properly convex domains in $\RP^2$.
Recall $\mathfrak{C}$ is not $\text{T}_0$.
These domains give rise to the only known examples with more than one point of minimal nonempty closed subsets of $\mathfrak{C}$ \cite{nolte2023leaves}. 

The bounds in Theorem C do not match because $T_\rho^k$ need not be $C^1$ ($T_\rho^2$ is never $C^1$ for $n=4$ \cite{nolte2024foliationsSl4}). This breaks a crucial aspect of known sharp regularity computations in similar settings.
Obtaining even these bounds requires new ideas.
What is missing from an exact expression for $D_B$ is an understanding of the structure of $\Phi_{ab}$ near $\xi^1_{\rho_b}(\partial \Gamma)$. 
In any matter, the following asymmetric completeness property of $D_B$ implies Theorem B.

\begin{theoremD}\label{thm-completeness}
For $n > 3$ the bouquet coupling metric $D_B$ is forward proper, in the sense that for any $\rho_0 \in {\rm{Hit}}_n(S)$ and $R > 0$ the forward ball $\{\rho \in {\rm{Hit}}_n(S) \mid D_B(\rho_0, \rho) \leq R\}$ is compact.
\end{theoremD}

\begin{remarks} \,
    \begin{enumerate}
       \item  We prove slightly more what we stated in Theorems A and C: we show the coupling maps {\rm{attain}} the regularity given by the lower bounds shown for H\"older exponents. For instance, for $\rho_a,\rho_b$ Fuchsian, the map $\phi_{ab}$ is shown to be $\exp(-d_{{\rm{Th}}}(\rho_a, \rho_b))$-H\"older, rather than  only showing $\phi_{ab}$ is $\alpha'$-H\"older for all $\alpha' < \exp(-d_{{\rm{Th}}}(\rho_a, \rho_b))$.
    \item Theorem C and standard results imply $D_B$ is a continuous metric on ${\rm{Hit}}_n(S)$. The relevant facts are that limit cones of Hitchin representations are contained in the interior of $\mathfrak{a}^+$ and vary continuously (\cite{potrie2017eigenvalues-entropy} Cor. 4.9, \cite{sambarino2014hyperconvex}), and that $d_{\rm{Th}}^{\alpha_k}$ is Finsler \cite{carvajales2022thurstonsasymmetric}.
\item     The bounds of Theorem C applied to the Fuchsian locus $\mathscr{F}_n \subset \rm{Hit}_n(S)$ show that for each $n$ there is a constant $C_n$ so that the restriction of $D_B$ to $\mathscr{F}_n$ is $C_n$-multiplicatively comparable to Thurston's asymmetric metric.
    So $D_B$ is genuinely asymmetric.
\item     $D_B$ is only defined for $n > 3$.
    The appropriate analogues for $n = 2, 3$ seem to be to use $d_{\rm{Th}}$ for $n = 2$ and $d_{\rm{Th}}^{\alpha_1} = d_{\rm{Th}}^{\alpha_2}$ for $n = 3$.
    For $n = 3$, the asymmetric metric $d_{\rm{Th}}^{\alpha_1}$ has the same completeness properties as $D_B$ for reasons that fail in all higher ranks.
\item    Theorem A here and Theorem 1.3 in \cite{carvajales2022thurstonsasymmetric} show $d_k$ is an asymmetric metric on ${\rm{Hit}}_n(S)$, except in the case where $n > 3$ is even and $k =n/2$.
    \end{enumerate}
\end{remarks}
    
\subsubsection{Flag Coupling}\label{sss-full-flag-bad}
We describe our coupling construction for limit curves in $\mathcal{F}(\bbR^n)$, and the technical complications the construction faces.
As this is non-central, we are brisk.

We expect that there are canonical parametrizations of \textit{larger} subspaces than the torus bouquets. 
The model here is the work of Guichard-Wienhard in \S 4.1 of \cite{guichard2008convex} for $n = 4$ in $\RP^3$, where canonical parameterizations of each connected component of $\RP^3 - \mathcal{B}_\rho$ are produced.
We expect the involved maps to quickly become complicated and involve a choice from a finite collection of canonical maps \cite{nolte2024foliationsSl4}.

\begin{definition}
    For $\rho_a, \rho_b \in {\rm{Hit}}_n(S)$, the {\rm{flag coupling map}} $\phi_{ab}^\mathcal{F} : \xi_b(\partial \Gamma) \to \xi_a(\partial\Gamma)$ is $\xi_a \circ \xi_b^{-1}$.
    Let $\alpha_{ab}^\mathcal{F}$ be the supremum of the $\alpha > 0$ so $\phi_{ab}^\mathcal{F}$ is $\alpha$-H\"older, and $d_\mathcal{F}(\rho_a, \rho_b) = - \log \alpha_{ab}^\mathcal{F}$.
\end{definition}

A similar argument to Theorems A and C above, which we omit, shows that
    \begin{align}
        \alpha_{ab}^\mathcal{F} = \min\left(1, \inf_{\gamma \in \Gamma - \{e\}} \frac{\min\limits_{k=1,...,n-1} \ell_{\rho_a}^{\alpha_k}(\gamma) }{\min\limits_{k=1,...,n-1} \ell_{\rho_b}^{\alpha_k}(\gamma)}\right).\label{eq-flag-couple}
    \end{align}
For $n > 3$, even showing that $\alpha_{ab}^\mathcal{F}$ is strictly less than $1$ if $\rho_a, \rho_b \in \text{Hit}_n(S)$ are distinct seems challenging.
The closest results to this are in \cite{potrie2017eigenvalues-entropy} and \cite{bridgeman2015pressure}, in particular their application to prove that $d_{\text{Th}}^{\alpha_k}$ separates points in \cite{carvajales2022thurstonsasymmetric}.

\subsection{Related Work} Coupling maps for Fuchsian representations go back to Nielsen \cite{nielsen1927untersuchungen}, and the study of their H\"older regularity to Ahlfors \cite{ahlfors1954quasiconformal}
(see also \cite{mori1956absolute}).
Theorem A consists of two bounds on H\"older regularity.
The lower bound is the difficulty. The upper bound was is due to Sorvali \cite{sorvali1973dilatation}, and the equality to Tsouvalas \cite{tsouvalas2023holder}.
Papadapolous-Su have noted the relationship of Sorvali's work and $d_{\text{Th}}$, in case of surfaces with at least one puncture \cite{papadopoulusSu2014Thurston}.

Numerous regularity computations on domains and mappings associated to Anosov representations have been carried out previously \cite{benoist2004convexesI} \cite{guichard2005regularity} \cite{sambarino2016entropy} \cite{tsouvalas2023holder} \cite{zhangZimmer2019Regularity}.
A crucial point in obtaining optimal results in such computations is Abels-Margulis-Soifer's work \cite{abelsMargulisSoifer1995semigroups} on proximality and Benoist's further development of the direction \cite{benoist1997proprietes}.
This technique appears in \cite{guichard2005regularity} and \cite{tsouvalas2023holder}.

As mentioned above, Carvajales-Dai-Pozzetti-Wienhard have produced other extensions of Thurston's asymmetric metric in the setting of Anosov representations \cite{carvajales2022thurstonsasymmetric}.
The perspectives of \cite{carvajales2022thurstonsasymmetric} and this work are different for Teichm\"uller space.
Pursuing the different perspectives leads to constructions and metrics of different flavors.
Namely, the starting point of \cite{carvajales2022thurstonsasymmetric} is a formula of Thurston for $d_{\text{Th}}$ \cite{thurston1998minimal} while our work views $d_{\text{Th}}$ as a measurement of the H\"older distortion of boundary maps.
Neither perspective directly generalizes Thurston's original definition of his metric in terms of Lipschitz distortion of hyperbolic metrics.

Carvajales-Dai-Pozzetti-Wienhard have the insight that involving the dynamics of Anosov representations allows for a compellingly general theory to be developed.
This perspective sees aspects of Anosov representations that are important in constructions of metrics and less clear from other viewpoints.
For instance, though we reproduce simple root stretch metrics, it is not at all clear how to prove that they separate points without the perspective of \cite{carvajales2022thurstonsasymmetric}. 

The main contrasts of our perspective with that of \cite{carvajales2022thurstonsasymmetric}, and what we see as the appeals of our viewpoint, are that our construction is more directly geometric and uses the full data of the Hitchin condition.
In particular, we use the full hyperconvexity of Hitchin representations \cite{guichard2008composantes} \cite{labourie2004anosov}.
Our construction is necessarily much less general than \cite{carvajales2022thurstonsasymmetric}.

Simple root metrics are distinguished among the asymmetric metrics of \cite{carvajales2022thurstonsasymmetric} as not requiring an entropy normalization in their definition, as Potrie-Sambarino proved the relevant entropy is constant \cite{potrie2017eigenvalues-entropy}.
It is appealing that these distinguished metrics have a geometric realization.

\subsection{Outline}\label{ss-outline} In \S \ref{s-simple-stretch} we deduce Theorem A from work of Tsouvalas \cite{tsouvalas2023holder}.
We also present a distinct and simplified proof in the case of Thurston's asymmetric metric in \S \ref{s-teich-space-case}.
This is a model for later arguments, and accessibile without background on Anosov representations.

The remainder of the paper concerns the bouquet coupling metrics.
The main objective is the upper bound of Theorem C.
There are three main technical matters in its proof to navigate in our setting, beyond the ones seen in the study of limit curves of Hitchin representations in Grassmannians or boundaries of divisible domains (e.g. \cite{guichard2005dualite} \cite{tsouvalas2023holder}).

The first difficulty is that the torus bouquet $\mathcal{B}_\rho$ contains \textit{all but at most one} eigenline of $\rho(\gamma)$ for $\gamma \in \Gamma$.
Techniques from quantitative proximality are essential to the estimate, and are effective in analyzing neighborhoods of attracting fixed-points of linear maps.
Application these techniques to the parts of $\mathcal{B}_\rho$ containing intermediate-eigenvalue eigenlines is hopeless.

The second difficulty is that $\mathcal{B}_\rho$ is not globally $C^1$ in any sense.
In the irregular region, near $\xi^1(\partial \Gamma)$, this makes estimation of distances between points problematic.

The third difficulty is that the natural space that $\Gamma$ acts cocompactly on, the space $\partial \Gamma^{(3)+}$ of positively oriented triples in $\partial \Gamma$, is of a lower dimension than $\mathcal{B}_\rho \times \mathcal{B}_\rho$.
So one application of the uniform convergence group property of $\partial \Gamma$ does not move an arbitrary pair of points in $\mathcal{B}_\rho$ to compact family of pairs, and further analysis is needed.

What makes these difficulties tractable is that on compact subsets of $\mathcal{B}_\rho - \xi^1(\partial \Gamma)$ each torus $T_\rho^k$ is $C^1$ and has an approximate-product structure compatible with the coupling maps $\Phi_{ab}$ (\S \ref{s-bouquet-coupling}).
This reduces the regularity of the coupling maps on compact subsets of $\mathcal{B}_\rho - \xi^1(\partial \Gamma)$ to the regularity of the $k$-coupling maps $\phi_{ab}^k$ between limit curves in $k$-Grassmannians.
This circumvents the first challenge and reduces analysis to a small neighborhood of $\xi^1(\partial \Gamma)$.

So we are left to consider the regularity of coupling maps on a neighborhood of $\xi^1(\partial \Gamma)$ of a fixed small scale of our choosing.
With our freedom to set this scale we arrange for all points in $\partial\Gamma$ needed to specify a pair of points in $\mathcal{B}_\rho$ to be close together, and for numerous other technical properties to hold.
This particular structure allows for the third challenge to be addressed with an adapted application of the uniform convergence group property of $\Gamma$.

This technical perspective of this work has novelties.
In particular, our use of control over scales under consideration to improve quality constants does not appear in prior regularity computations discussed above.
Additionally, our use of the uniform convergence group property of $\Gamma$ in terms of the action of $\Gamma$ on triples in $\partial \Gamma$ is a rather efficient source of uniformity in our computations that seems likely to generalize well to more complicated settings.

\medskip 

\par \noindent \textbf{Acknowledgements.} The author thanks Jean-Marc Schlenker for inspiring his interest in questions related to this paper, Mike Wolf for his support and interest, Aaron Calderon for stimulating conversations, and Christopher Bishop for helpful comments on the exposition. This material is based on work supported by the National Science Foundation under Grants No. 1842494 and 2005551.

\section{Simple Root Stretch Metrics}\label{s-simple-stretch}

\subsection{Teichm\"uller Space}\label{s-teich-space-case}
We give a proof of Theorem A in the case of $\mathcal{T}(S)$ here, before deducing the general case from \cite{tsouvalas2023holder}.
This is a model of the core of the proof of Theorem C.

\subsubsection{Proximality}
We briefly set notation and recall the standard results on proximal elements of $\text{PSL}_n(\bbR)$ that are essential to the main estimate.
Our notation follows \cite{guichard2005regularity}.

A linear map $g \in \text{PSL}_n(\bbR)$ is \textit{proximal} if $g$ fixes a line $x^+_g$ and transverse hyperplane $X^<_g$ so that $\lim_{n\to\infty} g^n x = x^+_g$ for all $x \notin X^<_g$.
With respect to a background Riemannian metric $\delta$ on $\RP^{n-1}$, we say that for $r > \varepsilon > 0$ the map $g$ is \textit{$(r,\varepsilon)$-proximal} if \begin{align*}
    \delta(x_g^+, X^<_g) \geq r, \qquad g(\mathbb{RP}^{n-1} - N_\varepsilon(X^<_g)) \subset N_\varepsilon(x^+_g), \qquad g|_{\RP^{n-1} - N_\varepsilon(X^<_g)} \text { is $\varepsilon$-Lipschitz.}
\end{align*} where $N_\varepsilon(S)$ is the $\varepsilon$-neighborhood of a set $S$ throughout the paper.
We say $g$ is $(r,\varepsilon)$\textit{-biproximal} if both $g$ and $g^{-1}$ are $(r,\varepsilon)$-proximal.
The consequence of \cite{abelsMargulisSoifer1995semigroups} needed here is:

\begin{fact}[\cite{abelsMargulisSoifer1995semigroups} 5.17, \cite{benoist1997proprietes} \S 3.1] \label{lemma-teich-proximality}
    Let $\rho_a, \rho_b: \Gamma \to {\rm{PSL}}_2(\bbR)$ be Fuchsian. Then for any $r> \varepsilon > 0$ sufficiently small there is a finite set $F \subset \Gamma$ so that for all $\gamma \in \Gamma$, there is an $\eta \in F$ so $\rho_a(\eta \gamma)$ and $\rho_b(\eta\gamma)$ are $(r,\varepsilon)$-proximal.
\end{fact}

\subsubsection{Regularity Computation}\label{ss-computation-teich} 
The relevant notations to this computation are as follows. For Fuchsian representations $\rho_a, \rho_b$ and $\gamma \in \Gamma - \{e\}$, define $r_{ab}(\gamma) = \ell_{\rho_a}(\gamma)/\ell_{\rho_b}(\gamma)$, and $h_{ab} = \inf_{\gamma \in \Gamma - \{e\}} r_{ab}(\gamma)$.
In \cite{thurston1998minimal}, Thurston proves $d_{\text{Th}}(\rho_a, \rho_b) = \log \, \sup_{\gamma \in \Gamma - \{e\}} r_{ab}^{-1}$.
Denote the limit maps $\partial \Gamma\to \partial \bbH$ of $\rho_a$ and $\rho_b$ by $\xi_a$ and $\xi_b$, respectively.
For $g \in \text{PSL}_2(\bbR)$ let $\ell(g)$ be translation length on $\bbH$.

\begin{theorem}[Coupling is Stretch]
    $d_1 = d_{{\rm{Th}}}$ and $\phi_{ab}$ attains its optimal H\"older exponent.
\end{theorem}

\begin{proof}
    We first show $d_1 \geq d_{\text{Th}}$. Let $\gamma \in \Gamma - \{e\}$ and $\alpha > r_{ab}(\gamma)$.
    Work in models of $\bbH$ with attractor $\xi_c(\gamma^+)$ at $0$ and $\xi_c(\gamma^-)$ at $\infty$ for $c \in \{a,b\}$.
    The induced metric on $\RP^1 -\{\infty\}$ is multiplicatively comparable to the reference metric on $\RP^1$ in the complement of any fixed neighborhood of $\infty$.
    Consider the sequence $(x_n, y_n) = (e^{-n\ell_{\rho_b}(\gamma)},0)$.
    Then
    \begin{align*}
        \frac{d(\phi_{ab}(x_n), \phi_{ab}(y_n))}{d(x_n,y_n)^\alpha} &= \frac{d(\rho_a(\gamma)^n\phi_{ab}(1), \rho_a(\gamma)^{n}\phi_{ab}(0))}{d(\rho_b(\gamma)^{n}1,\rho_b(\gamma)^{n}0)^\alpha} = e^{n (\alpha \ell_{\rho_b}(\gamma) - \ell_{\rho_a}(\gamma))} {d(\phi_{ab}(1), \phi_{ab}(0))}.
    \end{align*} As $\alpha > r_{ab}(\gamma)$, $\exp(n(\alpha \ell_b(\gamma) - \ell_a(\gamma))$ grows unboundedly.
    So $\phi_{ab}$ is not $\alpha$-H\"older.
    Hence $$d_1(\rho_a, \rho_b) = -\log \alpha_{ab} \geq -\log \left( \inf_{\gamma \in \Gamma - \{e\}} \frac{\ell_{\rho_a}(\gamma)}{\ell_{\rho_b(\gamma)}}\right) = \log \left( \sup_{\gamma \in \Gamma - \{e\}}\frac{\ell_{\rho_b}(\gamma)}{\ell_{\rho_a}(\gamma)} \right) = d_{\text{Th}}(\rho_a,\rho_b).$$

    We prove the other inequality by showing the above behavior dominates the critical cases for the H\"older condition.
    The key step is a use of Fact \ref{lemma-teich-proximality} to separate attracting and repelling fixed-points of maps under consideration in an application of the uniform convergence group property of $\Gamma$.
    The other matter to control is that the metrics on $\RP^1 - \{\infty\}$ induced from upper half plane models of $\bbH$ have poor comparability near $\infty$.
    
    Fix reference metrics $\delta$ on $\partial \bbH$ from the Poincar\'e model and $d_{\partial\Gamma}$ on $\partial \Gamma$.
    We will exploit that only nearby points are relevant to the H\"older condition since $\partial \Gamma$ is compact.
    Let $r > \varepsilon > 0$ be sufficiently small and fixed.
    By Fact \ref{lemma-teich-proximality}, there is a finite set $F \subset \Gamma$ so that for all $\gamma \in \Gamma$ there is an $\eta \in F$ so that $\rho_a(\eta\gamma)$ and $\rho_b(\eta\gamma)$ are $(r,\varepsilon)$-proximal.

    As $\Gamma$ is a uniform convergence group, there is a compact subset $K$ of the collection $\partial \Gamma^{(3)+}$ of positiviely oriented triples of elements of $\partial \Gamma$ so that for all $t \in \partial \Gamma^{(3)+}$ there is a $\gamma \in \Gamma$ so $\gamma t \in K$.
    Replacing $K$ with the compact set $L = F K$, with $$\Gamma_{(r,\varepsilon)}^{a,b} = \{ \gamma \in \Gamma \mid \rho_a(\gamma) \text{ and }\rho_b(\gamma) \text{ are } (r,\varepsilon) \text{-proximal}\},$$ for every $t \in \partial \Gamma^{(3)+}$ there is a $\gamma \in \Gamma_{(r,\varepsilon)}^{a,b}$ so $\gamma t \in L$.
    Compactness of $L$ ensures that there is a $\nu > 0$ so that for all $(x,y,z) \in L$ each of $\delta(\xi_a(x), \xi_a(y)), \delta(\xi_a(x),\xi_a(z))$, and $\delta(\xi_a(y), \xi_a(z))$ and their analogues for $\xi_b$ are at least $\nu$.

    We next use our freedom to choose the scale of points we consider to reduce the $\varepsilon$ for which we must consider $(r,\varepsilon)$-proximal maps \textit{while still working with the same compact $L$}.

    So let $x,y \in \partial \Gamma$ be distinct. We take $d(\xi_a(x),\xi_a(y))$ and $d(\xi_b(x), \xi_b(y))$ to be below a fixed small scale, that we shall describe.
    The first requirement is that $d_{\partial \Gamma}(x,y) < \text{diam}(\partial \Gamma)/3$.

    After potentially interchanging $x$ and $y$, we may choose $z \in \partial \Gamma$ so $(x,y,z)$ is positively oriented and both $d_{\partial \Gamma}(x,z), d_{\partial\Gamma}(y,z)$ are at least $\text{diam}(\partial \Gamma)/3$. 
    Then both $\xi_a(z)$ and $\xi_b(z)$ are separated by a $(x,y)$-independent constant $r_0$ from $\xi_a(x),\xi_a(y)$ and $\xi_b(x),\xi_b(y)$, respectively. 
    Now there exists a point $(x_0, y_0,z_0) \in L$ and $\gamma_{xy} \in \Gamma_{(r,\varepsilon)}^{a,b}$ so that $\gamma_{xy}(x,y,z) = (x_0,y_0,z_0)$.

\begin{claim}[Proximality Refinement]\label{claim-teich-better-biproximality}
    For all $\varepsilon' > 0$ there is a $\tau(\varepsilon') > 0$ so that for all distinct $x,y$ in $\partial \Gamma$ with $d_{\partial \Gamma}(x,y) < \tau(\varepsilon')$ both $\rho_a(\gamma_{xy})$ and $\rho_b(\gamma_{xy})$ are $(r, \varepsilon')$-biproximal.
\end{claim}

\begin{proof}
    Fix $r > 0$.
    Straightforward computation in the Poincar\'e model of $\bbH$ shows the following two observations.\footnote{A detailed argument in greater generality is given in \S \ref{ss-loxodromy-refinement}.}
    First, for all $\varepsilon' > 0$ there is an $T < \infty$ so if $g \in \text{PSL}_2(\bbR)$ is proximal, has $\delta(x_g^+, X_g^<) \geq r$, and has $\ell(g) > T$ then $g$ is $(r,\varepsilon')$-proximal.
    In fact, symmetry of this condition under inversion allows a conclusion of $(r,\varepsilon')$-biproximality.
    Second, for all $T > 0$ there is a $B = B(T,r)$ so for every proximal $g \in \text{PSL}_2(\bbR)$ with $\delta(x_g^+, X_g^<) \geq r$ and $\ell(g) \leq T$ the action of $g$ on $\partial \bbH$ is $B$-bilipschitz.
    
    The optimal bilipschitz factor of $\rho_c(\gamma_{xy})$ $(c \in \{a,b\})$ is bounded below by $\nu/\delta(\xi_c(x), \xi_c(y))$.
    So take $\tau(\varepsilon')$ so that if $d_{\partial\Gamma}(x,y) < \tau(\varepsilon')$ then both $\delta(\xi_a(x),\xi_a(y))$ and $\delta(\xi_b(x),\xi_b(y))$ are less than a constant $\eta(\varepsilon')$ so small that the bilipshitz factor lower bound of $\nu/\eta(\varepsilon')$ for the action of $\rho_c(\gamma_{xy})$ on $\partial \bbH$ implies $(r, \varepsilon')$-biproximality of $\rho_c(\gamma_{xy})$ ($c \in \{a,b\}$).
\end{proof}

Now let $\varepsilon' < \text{min}(r_0/6,\nu/6)$. We restrict our attention further to $(x,y)$ with $\delta(\xi_a(x),\xi_a(y))$ and $\delta(\xi_b(x),\xi_b(y))$ less than $\text{min}(\tau(\varepsilon'), \varepsilon'/3)$, where $\tau(\varepsilon')$ is provided by the preceding claim.

So fix such a pair $(x,y)$ and fix $c \in \{a,b\}$. 
We record how the $(r,\varepsilon')$-proximality of $\rho_c(\gamma_{xy})$ constrains the location of $\gamma_{xy}^+$ and $\gamma_{xy}^-$ relative to the other points.
We first note that since $\rho_c(\gamma_{xy})$ is $(r,\varepsilon')$-proximal, at least one of $\xi_c(x)$ and $\xi_c(y)$ must be in $N_{\varepsilon'}(\xi_c(\gamma_{xy}^-))$: otherwise both $\xi_c(\gamma_{xy} x)$ and $\xi_c(\gamma_{xy} y)$ would be in $N_{\varepsilon'}(\xi_c(\gamma_{xy}^+))$ which is impossible because $\varepsilon' < \nu$.
So both $\xi_c(x)$ and $\xi_c(y)$ are contained in $N_{2\varepsilon'}(y)$.

Observe also that since $\delta(\xi_c(x), \xi_c(z)) > 6\varepsilon'$ that $\xi_c(z) \notin N_{\varepsilon'}(\gamma_{xy}^-)$.
So $\rho_c(\gamma_{xy} z) \in N_{\varepsilon'}(\gamma_{xy}^+)$ by $(r,\varepsilon')$-proximality of $\rho_c(\gamma_{xy})$.
This implies $\xi_c(\gamma_{xy}x)$ and $\xi_{c}(\gamma_{xy}y)$ are separated from $\gamma_{xy}^+$ by at least $\nu/2$.
We finally remark that for at least one $w_c \in \{x,y\}$ the point $\xi_c(\gamma_{xy} w_c)$ is separated from $\xi_c(\gamma_{xy}^-)$ by at least $\nu/2$, by construction of $\nu$.

Now, for each $c \in \{a,b\}$ work in an upper-half-space model of $\bbH$, normalized so that $\xi_c(\gamma_{xy}^-)$ is at $0$, $\xi_c(\gamma_{xy}^+)$ is at $\infty$, and $\xi_c(\gamma_{xy} w_c)$ is at $\pm 1$, with sign determined by the orientation of $(\gamma_{xy}^-, \gamma_{xy}w_c, \gamma_{xy}^+)$.
From the uniform separation of the normalization points, the metrics $d_c$ on $\partial \bbH$ in these models are uniformly multiplicatively comparable to $\delta$ on the complements of $\nu/2$-neighborhoods of $\xi_c(\gamma_{xy}^+)$ (taken with respect to $\delta$).

In these metrics,
\begin{align*}
 \frac{d_a(\phi_{ab}(\xi_b(x)), \phi_{ab}(\xi_b(y)))}{d_b(\xi_b(x),\xi_b(y))^{h_{ab}}} &=  \frac{e^{-\ell_a(\gamma_{xy}^{-1})}d_a(\phi_{ab}(\xi_b(x_0)), \phi_{ab}(\xi_b(y_0)))}{e^{-h_{ab}\ell_b(\gamma_{xy}^{-1})}d_b(\xi_b(x_0),\xi_b(y_0))^{h_{ab}}} \\
 &\leq \exp{(h_{ab} \ell_b(\gamma_{xy}^{-1}) - \ell_a(\gamma_{xy}^{-1}))} C \\
 & \leq C,
\end{align*} for a fixed finite $C$ supplied by the compactness of $L$.
So $h_{ab} = \alpha_{ab}$ and $\phi_{ab}$ is $h_{ab}$-H\"older and
\begin{align*} d_1(\rho_a, \rho_b) = -\log \alpha_{ab} \leq -\log \left( \inf_{\gamma \in \Gamma - \{e\}} \frac{\ell_{\rho_a}(\gamma)}{\ell_{\rho_b(\gamma)}}\right) = \log \left( \sup_{\gamma \in \Gamma - \{e\}}\frac{\ell_{\rho_b}(\gamma)}{\ell_{\rho_a}(\gamma)} \right) = d_{\text{Th}}(\rho_b,\rho_a). &\qedhere
\end{align*}
\end{proof}

\subsection{Notations and reminders}
We set some notation here for the rest of the paper.

\subsubsection{Lie-Theoretic Notations}\label{sss-lie-notations}
Let $\mathfrak{a} = \{ (x_1, ..., x_n) \in \bbR^n \mid \sum_{i=1}^n x_i = 0 \}$ be a standard Cartan subalgebra of $\mathfrak{sl}_n(\bbR)$, identified with diagonal trace-free matricies in $\mathfrak{gl}_n(\bbR)$.
We fix the Weyl chamber $\mathfrak{a}^+ = \{(x_1, ..., x_n) \in \mathfrak{a}  \mid x_1 \geq ... \geq x_n \}$. Denote the interior of $\mathfrak{a}^+$ by $\mathring{\mathfrak{a}}^+$. 
Let $\lambda : \text{PSL}_n(\bbR) \to \mathfrak{a}^+$ be the Jordan projection, i.e. $\lambda(g) = (\ell_1(g), ... , \ell_n(g))$ is the logarithms of the absolute values of the generalized complex eigenvalues of $g$, listed in non-increasing order with multiplicity.

For $\varphi \in \mathfrak{a}^*$ so that the restriction of $\varphi$ to $\mathfrak{a}^+$ is non-negative, the \textit{$\varphi$-length} $\ell^\varphi : \text{PSL}_n(\bbR) \to \bbR^+$ is given by $\ell^\varphi(g) = \varphi(\lambda(g))$. We adopt the notation that for a representation $\rho : \Gamma \to \text{PSL}_n(\bbR)$ and $\gamma \in \Gamma$ that $\ell_\rho^\varphi(\gamma) = \ell^\varphi(\rho(\gamma))$.

For $i = 1, ..., n-1$ write the $i$-th simple root $\mathfrak{a} \to \bbR$ given by $(x_1, ..., x_n) \mapsto x_i - x_{i+1}$ as $\alpha_i$.
The \textit{Hilbert functional} $\text{H}: \mathfrak{a} \to \bbR$ is given by $\text{H}(x_1, ..., x_n) = x_1 - x_n = \sum_{i=1}^{n-1} \alpha_i$.

The \textit{limit cone} $\mathcal{L}_\Lambda$ of a subgroup $\Lambda < \text{PSL}_n(\bbR)$ is the smallest closed cone containing $\lambda(g)$ for all $g \in \Lambda$ \cite{benoist1997proprietes}. Denote the limit cone of the image $\rho(\Gamma)$ of a representation $\rho$ by $\mathcal{L}_\rho$.

\subsubsection{Hitchin Representations and Hyperconvexity}\label{sss-hyperconvexity}
We recall the characterization of Hitchin representations in terms of the structure of their boundary maps and some related facts.

Informally, the relevant condition on a map $\xi: \partial \Gamma \to \mathcal{F}(\bbR^n)$ is that subspaces formed as sums of entries of $\xi(x)$ are in as general position as possible, and limits of sums of flag entries for nearby points converge to higher entries of the flag.
Formally, a continuous map $\xi: \partial \Gamma \to \mathcal{F}(\bbR^n)$ is a \textit{hypercovnex Frenet curve} if:
\begin{enumerate}
    \item For any $k_1,...,k_j$ so $\sum_{l=1}^j k_l = p \leq n$, and distinct $x_1,..., x_j \in \partial \Gamma$, the sum $\xi^{k_1}(x_1) + ... + \xi^{k_j}(x_j)$ is direct.
    \item For any $x \in \partial \Gamma$, $k_1,...,k_j$ as above, and sequence $(x_1^m, ..., x_j^m)$ of $j$-ples of distinct points in $\partial \Gamma$ converging to $(x, ..., x)$ we have that $\xi^p(x) = \lim\limits_{m \to \infty} \bigoplus_{l = 1}^j \xi^{k_l}(x_l^m)$.
\end{enumerate}

This condition implies far more families of subspaces than those in the definition must be in general position and satisfy limit compatibility conditions, as is shown in \cite{guichard2005dualite}. We use the results of \cite{guichard2005dualite} in the following, especially Th\'eor\`eme 2, Lemme 6, and Prop. 7. We have:

\begin{theorem}[\cite{guichard2008composantes} \cite{labourie2006anosov}] A representation {\rm{$\rho: \Gamma \to \text{PSL}_n(\bbR)$}} is Hitchin if and only if there exists a $\rho$-equivariant hyperconvex Fren\'et curve $\xi_\rho : \partial \Gamma \to \mathcal{F}(\bbR^n)$.
\end{theorem}

The structure of exterior powers of Hitchin representations is relevant in the sequel.
The definition here is due to Pozetti-Sambarino-Wienhard \cite{pozzettiSambarino2022lipschitz}: a representation $\rho: \Gamma \to \text{PSL}_n(\bbR)$ is \textit{(1,1,2)-hyperconvex} if it is $\{P_1, P_2\}$-Anosov (see \cite{guichard2012anosov}) and $(\xi^1(x) \oplus \xi^1(y)) \cap \xi^{n-2}(z) = \{0\}$ for all distinct $x,y,z \in \partial \Gamma$.
The Anosov condition gaurantees the existence of limit maps $(\xi^1, \xi^2, \xi^{n-2}, \xi^{n-1})$ of partial flags consisting of subspaces of the dimensions labelled by superscripts.
All that we shall need on $(1,1,2)$-hyperconvexity is contained in \cite{pozzettiSambarinoWienhard2021conformality}:
\begin{enumerate}
    \item If $\rho: \Gamma \to \text{PSL}_n(\bbR)$ is Hitchin, then $\Lambda^k \rho$ is $(1,1,2)$-hyperconvex for $k = 1, ..., n-1$,
    \item If $\rho : \Gamma \to \text{PSL}_n(\bbR)$ is $(1,1,2)$-hyperconvex, then $\xi^1(\partial \Gamma) \subset \RP^{n-1}$ is $C^1$ with tangent line $\xi^2(x)$ at $\xi^1(x)$ for all $x \in \partial \Gamma$. 
\end{enumerate}

\subsection{General Simple Root Stretch Metrics}\label{s-simple-roots-computation}

We explain how Theorem A follows from work of Tsouvalas.
The observation made here almost appears in \S 6 of \cite{tsouvalas2023holder}, in particular around Corollary 6.1.
His theorem, stated in our notation, is:

\begin{theorem}[Tsouvalas \cite{tsouvalas2023holder} Thm. 1.9]
    Let $\rho_a$ and $\rho_b$ be $(1,1,2)$-hyperconvex representations to ${\rm{PSL}}_n(\bbR)$.
    Then $\psi_{ab}: \xi^1_b(\partial \Gamma) \to \xi^1_a(\partial \Gamma)$ defined by $\xi^1_a \circ (\xi^1_b)^{-1}$ realizes its supremal $\alpha$-H\"older regularity, with H\"older exponent $\alpha_{ab} = \inf_{\gamma \in \Gamma - \{e \}} \ell_{\rho_a}^{\alpha_k}(\gamma)/\ell_{\rho_b}^{\alpha_k}(\gamma).$
\end{theorem}

Write the first entry of the limit map of $\Lambda^n\rho$ as $\xi^1_{\Lambda^k\rho}$. To describe $\xi^1_{\Lambda^k\rho}$, denote the Pl\"ucker embedding $\text{Gr}_k(\bbR^n) \to \bbP(\Lambda^k \bbR^n)$ by $\text{Pl}_{n,k}$. Then $\xi^1_{\Lambda^k \rho} = \text{Pl}_{n,k} \circ \xi^k$ with $\xi = (\xi^1, ..., \xi^{n-1})$ the limit map of $\rho$.
The coupling construction is:
\begin{definition}
        For $\rho_a, \rho_b : \Gamma \to {\rm{PSL}}_n(\bbR)$ Hitchin, define $\psi_{ab}^k : \xi^1_{\Lambda^k\rho_b}(\partial \Gamma) \to \xi^1_{\Lambda^k\rho_a}(\partial \Gamma)$ by $\xi^1_{\Lambda^k\rho_a} \circ (\xi^1_{\Lambda^k\rho_b})^{-1} = {\rm{Pl}}_{n,k} \circ \phi_{ab}^k \circ {\rm{Pl}}_{n,k}^{-1}.$
    \end{definition}
The H\"older regularities of $\phi_{ab}^k$ and $\psi_{ab}^k$ agree because Pl\"ucker embeddings are smooth embeddings. From the behavior of eigenvalues of diagonal matricies under exterior powers and Tsouvalas' theorem we conclude $\alpha_{ab}^k = \inf_{\gamma \in \Gamma - \{e \}} \ell_{\rho_a}^{\alpha_k}(\gamma)/\ell_{\rho_b}^{\alpha_k}(\gamma)$ for Hitchin representations $\rho_a, \rho_b: \Gamma \to \text{PSL}_n(\bbR)$ and $k = 1, ..., n-1$.
Theorem A follows:
\begin{align*}
    d_k(\rho_a, \rho_b) = -\log \alpha_{ab}^k = - \log \left( \inf_{\gamma \in \Gamma - \{e \}} \frac{\ell_{\rho_a}^{\alpha_k}(\gamma)}{\ell_{\rho_b}^{\alpha_k}(\gamma)} \right) = \log \left( \sup_{\gamma \in \Gamma - \{e\}} \frac{\ell_{\rho_b}^{\alpha_k}(\gamma)}{\ell_{\rho_a}^{\alpha_k}(\gamma)} \right) = d_{\text{Th}}^{\alpha_k}(\rho_a, \rho_b).
\end{align*}

\section{Bouquets}\label{s-bouquet-coupling}
This section develops the structure of the tori $T_\rho^k$ used for the estimates in \S \ref{s-bounds}.
The key points are that each $T_\rho^k$ has two canonical foliations by Frenet curves, that $T_\rho^k - \xi^1_\rho(\partial \Gamma)$ is $C^1$, and that $T_\rho^k - \xi^1(\partial \Gamma)$ has an approximate-product structure.

\subsection{First Features}\label{ss-bouquet-basics}
Let $\rho : \Gamma \to \text{PSL}_n(\bbR)$ be a Hitchin representation with $n > 3$.
We recall the notations that the limit map of $\rho$ is $\xi = (\xi^1, ..., \xi^{n-1})$, that $\mathcal{K}_n = \{2, ..., \lfloor n/2 \rfloor\}$, that $k_n = \lfloor n/2 \rfloor$, where $\lfloor \cdot \rfloor$ is the floor function, that $B_n = (\partial \Gamma^2_{1} \sqcup ... \sqcup \partial \Gamma^2_{k_n})/\sim$ is the $n$-bouquet defined in the introduction, and $\Phi_\rho^k: \partial \Gamma^2_k \to \RP^{n-1}$ are the individual-$\partial \Gamma^2$ maps.
The maps $\Phi_\rho^k$ agree on all equivalent points and so induce a map $\Phi_\rho : B_n \to \RP^{n-1}$, with image denoted by $\mathcal{B}_\rho$.
Denote the set of pairs of distinct elements in $\partial \Gamma$ by $\partial \Gamma^{(2)}$ and the collection of triples of positively oriented elements of $\partial \Gamma$ by $\partial \Gamma^{(3)+}$.

\begin{lemma}\label{lemma-boquet-param}
    The map $\Phi_\rho : B_n \to \RP^{n-1}$ is a homeomorphism onto $\mathcal{B}_\rho$.
\end{lemma}

\begin{proof}
    Since $B_n$ is compact and Hausdorff, it suffices to show that $\Phi_\rho$ is a continuous injection.
    Continuity on each individual $\partial \Gamma^2_{k}$ ($k = 1,..., k_n$) is a direct consequence of the limit compatibility property of hyperconvex Frenet curves from their characterization in terms of intersections (\cite{guichard2005dualite} Lemme 6).
    This implies $\Phi_\rho$ is continuous.
    Injectivity is from the general position property of hyperconvex Frenet curves for intersections (\cite{guichard2005dualite} Lemme 6).
\end{proof}

The following lemma is quite useful in understanding the structure of $\mathcal{B}_\rho$.

\begin{lemma}[Frenet Restriction]\label{lemma-restriction}
Let $\xi : \partial \Gamma \to \mathcal{F}(\bbR^n)$ be a hyperconvex Frenet curve. Fix $1 < D < n$ and $t_0 \in \partial \Gamma$. Then $\xi_{t_0, D} = (\xi^1_{t_0, D}, ..., \xi^{D-1}_{t_0,D}): \partial \Gamma \to \mathcal{F}(\xi^D(t_0))$ defined by 
$$ \xi^k_{t_0, D}(s) = \begin{cases} \xi^{n-D+k}(s) \cap \xi^D(t_0) & s \neq t_0 \\ \xi^{k}(t_0) & s = t_0
 \end{cases} $$ is a hyperconvex Frenet curve.
\end{lemma}

\begin{proof}
    This follows from the characterization of hyperconvex Frenet curves in terms of intersections (\cite{guichard2005dualite} Lemme 6).
    Namely, if $k_1, ..., k_j$ are positive integers with $\sum_{i=1}^jk_i = D$ and $x_1, ..., x_j \in \partial \Gamma - \{t_0\}$ are distinct, since $\xi$ is a hyperconvex Frenet curve
    \begin{align} \bigcap_{i=1}^j \xi_{t_0, D}^{D-k_i} (x_i) \subset \xi^D(t_0) \cap \left( \bigcap_{i=1}^j \xi^{n-k_i}(x_i) \right) = \{0\}. \label{eq-computation-hyperconvexity} \end{align}
    Note that this implies that the inclusion in Equation \ref{eq-computation-hyperconvexity} is an equality.
    The case where one $x_j \in \partial \Gamma$ is equal to $t_0$ is similar, so that hyperconvexity is established.
    
    Limit compatibility is a consequence of the limit compatibility of $\xi$ as follows.
    Let $\sum_{i=1}^j k_i = k < D$ and let $(x_1^m, x_2^m, ..., x_j^m)$ be a sequence of $j$-ples of distinct entries of $\partial \Gamma$ converging to $(x,..., x)$ for some $x \in \partial\Gamma$.
    If $x \neq t_0$, by the general position property for $\xi_{t_0,D}$ proved above, transversality of $\xi^D(t_0)$ and $\xi^{n-k}(x)$, and limit compatibility of $\xi$, \begin{align*}
        \lim_{m \to\infty} \bigcap_{i=1}^j \xi^{D-k_i}_{t_0,D}(x_i^m) = \lim_{n \to \infty} \xi^D(t_0) \cap\left( \bigcap_{i=1}^j \xi^{n-k_i}(x_i^m) \right) = \xi^D(t_0) \cap \xi^{n-k}(x) = \xi^{D-k}_{t_0,D}(x).
    \end{align*}
    The case in which $t_0 = x$ is similar, with limit compatibility of $\xi$ playing the role of transversality of $\xi^D(t_0)$ and $\xi^{n-k}(x)$ in the above.
    \end{proof}

We remark that due to the Frenet Restriction Lemma, natural objects associated to Hitchin representations in projective spaces have inductive aspects.
Structures from lower rank reappear in intersections with subspaces, but without natural group actions.
See \S \ref{ss-foliations-sliding}.

\subsection{Regularity} The tori $T_\rho^k$ have complicated behavior near the image of $\xi^1$, but much better controlled behavior elsewhere.
The first relevant definition is:

\begin{definition}
    Let $\Xi^1_\rho \subset T_\rho^k$ denote $\xi^1(\partial \Gamma)$ and $\Xi^2_{k} = T_\rho^k - \Xi^1_\rho$. Let $\Xi^2_\rho \subset \mathcal{B}_\rho$ denote $\mathcal{B}_\rho - \Xi^1_\rho$.   
\end{definition}

In this subsection, we prove a $C^1$-regularity proposition for $\Xi^2_\rho$ and a local product structure lemma that is useful in \S \ref{s-bounds}.
The sources of our control over regularity are the hyperconvexity of $\xi$ and that the curves $\xi^k$ ($k = 1,..., n-1$) are $C^1$, considered as subsets of $\text{Gr}_k(\bbR^n)$.

\begin{proposition}[Regular off Frenet]\label{prop-regular-off-frenet}
    For $\rho \in {\rm{Hit}}_n(S)$, any $p \in \Xi^2_k$ is a $C^1$ point of $T_\rho^k$.
    
    Writing $p = \xi^k(x) \cap \xi^{n-k+1}(y)$ for $x \neq y$ in $\partial \Gamma$, the tangent space is given by
    \begin{align*}
        {\rm{T}}_p T_{\rho}^k = (\xi^k(x) \cap \xi^{n-k+2}(y)) + (\xi^{k+1}(x) \cap \xi^{n-k+1}(y)).    
    \end{align*}
\end{proposition}

We adopt the convention $\xi^n(x) = \bbR^n$ for the edge case $k = 2$.

\begin{remark} The restriction to $\Xi^2_k$ is essential. In the case $k = 2$ and $n = 4$, a corollary of \cite{nolte2024foliationsSl4} is that $\Xi^2_k$ is exactly the set of $C^1$ points of $T_\rho^2$.
\end{remark}

\begin{remark} Particular care is needed in the analysis and exact statements of regularity results in the following, as many closely related natural maps to those we consider fail to be $C^1$.

For instance, it is a consequence of Benoist's Limit Cone Theorem \cite{benoist1997proprietes} that with respect to the $C^1$ structure on $\Xi^2_k$ from its embedding in $\RP^{n-1}$ and any product $C^1$ structure on $\partial \Gamma^{(2)}$ the maps $\Phi_\rho^k : \partial \Gamma^{(2)} \to \Xi^2_k$ are never $C^1$ for non-Fuchsian Hitchin representations.
\end{remark}

\subsubsection{Differential of Intersection} In this paragraph, we recall the formula for the differential of the intersection of two subspaces in general position.

For a finite-dimensional real vector space $V$ and $k, l < \text{dim}(V)$ with $k + l > \text{dim}(V)$, let $$\mathscr{T}_{k,l}(V) = \{(U, W) \in \text{Gr}_k(V) \times \text{Gr}_l(V) \mid U + W = V \}$$ and let $\pitchfork : \mathscr{T}_{k,l} \to \text{Gr}_{l+k-n}(V)$ be the intersection map: $\pitchfork(U, W) = U \cap W$.
Recall the canonical model for $\text{T}_W \text{Gr}_k(V)$ as $\text{Hom}(W, V/W)$.
For $(U, W) \in \mathscr{T}_{k,l}(V)$ and $(\phi, \psi) \in \text{T}_{(U,W)}\mathscr{T}_{k,l}(V)$, we define $\{\phi + \psi \} \in \text{Hom}(U \cap W, V/(U \cap W)) = \text{T}_{U \cap W}\text{Gr}_{k+l-n}(V)$ as follows.

Since $k + l > n$, any $v \in V$ may be written as $v = a(v) + b(v)$ with $a(v) \in U$ and $b(v) \in W$.
Each of $a(v)$ and $b(v)$ are defined up to addition of elements of $U \cap W$.
Pick arbitrary lifts $\overline{\phi}: U \to V$ of $\phi$ and $\overline{\psi} : W \to V$ of $\psi$.
Note $b\left(\overline{\phi}(v)\right)$ and $a\left(\overline{\psi}(v)\right)$ ($v \in V$) are well-defined up to addition of elements of $U \cap W$.
So a map $\{\phi + \psi \} : U \cap W \to V/(U \cap W)$ is specified by $\{\phi + \psi \}(v) = \left[b\left(\overline{\phi}(v)\right) + a\left(\overline{\psi}(v)\right) \right]$ ($v \in U \cap W$).
Here, $[\cdot ]$ is equivalence modulo $U \cap W$.

\begin{lemma}[Intersection Derivative]\label{lemma-intersection-der}
    For all $(U, W) \in \mathscr{T}_{k,l}(V)$ and $(\phi, \psi) \in {\rm{T}}_{(U,W)}\mathscr{T}_{k,l}(V)$,
    \begin{align}
        D_{(U,W)}\pitchfork (\phi, \psi) = \{ \phi + \psi\}.
    \end{align}
\end{lemma}
The proof is an exercise in differential geometry.
One solution takes the derivatives of projective duality and direct sum as maps on Grassmannians then applies the chain rule.

\subsubsection{Tangent Lines to Limit Curves} We now give an expression for the tangent line of the image of the limit map $\xi^k(\partial \Gamma) \subset \text{Gr}_k(\bbR^n)$ of a Hitchin representation $\rho$.
The proof is by a reduction to \cite{pozzettiSambarinoWienhard2021conformality}. We use the group invariance of the Frenet curve throughout the paragraph.

Let $\bbP \text{T} \text{Gr}_k(\bbR^n)$ be the projective tangent bundle of $\text{Gr}_k(\bbR^n)$ with projection $\pi_k$ to $\text{Gr}_k(\bbR^n)$ and define $$\mathscr{L}_\rho^k = \{ (x, \phi) \in \partial \Gamma \times \bbP \text{T} \text{Gr}_k(\bbR^n) \mid \pi_k(\phi) = \xi^k(x), \, \xi^{k-1}(x) = \text{ker} \, \phi, \, \text{Im}(\phi) \subset \xi^{k+1}(x)\}.$$
As the target $\xi^{k+1}(x)/\xi^k(x)$ of $\phi$ is one-dimensional, for each $x \in \partial \Gamma$ there is a unique point of the form $(x, \phi) \in \mathscr{L}_\rho^k$. Note $\mathscr{L}_\rho^k$ is closed in $\partial \Gamma \times \bbP \text{T}\text{Gr}_k(\bbR^n)$.

\begin{proposition}[Tangent Lines]
    For any $x\in \partial \Gamma$, the tangent line $\bbP{\rm{T}}_{\xi^k(x)}\xi^k(\partial \Gamma)$ is the unique $\phi_x \in \bbP {\rm{T}}_{\xi^k(x)}{\rm{Gr}}_k(\bbR^n)$ so $(x, \phi_x) \in \mathscr{L}_\rho^k$.
\end{proposition}

\begin{proof}
    Since $\xi^k(\partial \Gamma)$ is an embedded $C^1$ submanifold of $\text{Gr}_k(\bbR^n)$ and $\mathscr{L}_\rho^k$ is closed, it suffices to prove the claim for the dense set of $\partial \Gamma$ consisting of attracting fixed points $\gamma^+$ of $\gamma \in \Gamma - \{e\}$.

    So let $\gamma \in \Gamma - \{e\}$, and let $v_1, ..., v_n$ denote a basis of eigenvectors of $\rho(\gamma)$, ordered by decreasing modulus of the corresponding eigenvalues $\lambda_1, \lambda_2, ..., \lamdba_n$.
    Then $\xi^k(\gamma^+)$ is the span of $v_1, ..., v_k$.
    Let $\Lambda^k \rho$ denote the $k$-th exterior power of $\rho$, which is $(1,1,2)$ hyperconvex.
    Denote the first two entries of the associated limit curve by $\xi^1_{\Lambda^k}, \xi^2_{\Lambda^k}$.
    By \cite{pozzettiSambarinoWienhard2021conformality} Prop. 7.4, $\xi^1_{\Lambda^k}(\partial \Gamma)$ is an embedded $C^1$ submanifold of $\bbP(\Lambda^k \bbR^{n})$ with $ \text{T}_{\xi^{1}_{\Lambda^k}(x)}\xi^1(\partial \Gamma) = \xi^2_{\Lambda^k}(x)$ for all $x \in \partial \Gamma$.

    By the dynamics preservation of limit maps of Anosov representations,
    \begin{align}
        \xi^1_{\Lambda^k}(\gamma^+) &= \text{span} \{ v_1 \wedge v_2 \wedge \cdots \wedge v_k\}, \\
        \xi^2_{\Lambda^k}(\gamma^+) &= \text{span} \{v_1 \wedge v_2 \wedge \cdots \wedge v_{k}, v_1 \wedge v_2 \wedge \cdots \wedge v_{k-1} \wedge v_{k+1}\}. \nonumber
    \end{align}

    Since the Pl\"ucker embedding $\text{Pl}_{k,n}$ is a smooth embedding and $\xi^1_{\Lambda^k}(x) = \text{Pl}_{n,k}(\xi^k(x))$ for $x \in \partial \Gamma$, to establish the claim it suffices to show $D_{\xi^k(\gamma^+)}\text{Pl}_{n,k}(\phi_{\gamma^+})$ is contained in $\xi^2_{\Lambda^k}(\gamma^+)$.

    A representative of the projective class of $\phi_{\gamma^+}$ is given on the basis $v_1, ..., v_n$ of $\bbR^n$ by $\phi_{\gamma^+}(v_i) =\delta_{k, k+1} v_j$ in Kronecker $\delta$-notation.
    A representative $\psi^k_{\gamma^+}$ of the equivalence class of $\xi^2_{\Lambda^k}(\gamma^+) \in \bbP \text{T}_{\xi^1_{\Lambda^k}(x)} \bbP(\Lambda^k \bbR^n)$ is given on the standard basis $\{w_{I}\}$ indexed by positively ordered $k$-element subsets $I \subset (1, ...,n)$ of $\Lambda^k \bbR^n$ induced by the basis $v_1, ..., v_{n}$ of $\bbR^n$ by $\psi^k_{\gamma^+} w_{I} = \delta_{I_1 I_{2}} w_J$ where $I_1 = (1, ..., k)$ and $I_2 = (1, ..., {k-1}, {k+1})$.

    Now, $A_t = \text{span}(v_1, ..., v_{k-1}, v_k + t v_{k+1})$ in $\text{Gr}_k(\bbR^n)$ is tangent to $\phi_{\gamma^+}$ at $t= 0$.
    For $t > 0$,
    \begin{align*}
        \text{Pl}_{n,k}(A_t) = \text{span}\{ (v_1 \wedge v_2 \wedge \cdots \wedge v_k) + t(v_1 \wedge v_2 \wedge \cdots \wedge v_{k-1} \wedge v_{k+1}) \}
    \end{align*} is tangent to $\xi^2_{\Lambda^k}(\gamma^+)$, as desired.
\end{proof}

\subsubsection{Ambient Regularity} We now address the regularity of $\Xi^2_\rho$.

\begin{proof}[Proof of Prop. \ref{prop-regular-off-frenet}]
    Let $\mathscr{D}_{\rho}^k \subset \mathscr{T}_{k,n-k+1}$ be the submanifold of $\xi^k(\partial \Gamma)\times \xi^{n-k+1}(\partial \Gamma) $ given by \begin{align*}
       \mathscr{D}_\rho^k =  \xi^k(\partial \Gamma)\times \xi^{n-k+1}(\partial \Gamma) - \{(\xi^k(x), \xi^{n-k+1}(x)) \mid x \in \partial \Gamma \},
    \end{align*} and let $\pitchfork_{\rho}^k$ be the restriction of $\pitchfork$ to $\mathscr{D}_\rho^k$. 
    Then $\Xi^2_k$ is the image of $\pitchfork_{\rho}^k$. 
    As $\pitchfork$ is smooth on $\mathscr{T}_{k,n-k+1}$ and $\mathscr{D}_{\rho}^k$ is $C^1$, to show that $p = \xi^k(x) \cap \xi^{n-k+1}(y)$ $((x,y) \in \partial \Gamma^{(2)})$ is a $C^1$ point of $T_\rho^k$, it suffices to show that the differential $D_{(\xi^k(x), \xi^{n-k+1}(y))} \pitchfork_{\rho}^k$ has rank $2$.

    Given $(x, y) \in \partial \Gamma^{(2)}$ with corresponding points $(x, \phi) \in \mathscr{L}_\rho^k$ and $(y, \psi) \in \mathscr{L}_{\rho}^{n-k+1}$, the image of $D_{(\xi^k(x), \xi^{n-k+1}(y))} \pitchfork_{\rho}^k$ is spanned by $\{\phi\}$ and $\{\psi\}$.
    Note $\{\phi\}$ and $\{\psi\}$ each have rank at most $1$ as $\phi$ and $\psi$ have rank $1$.
    By hyperconvexity of $\xi$, \begin{align*}
        \ker \phi \cap \xi^{n-k+1}(y) &= \xi^{k-1}(x) \cap \xi^{n-k+1}(y) = \{0\}, \\
        \ker \psi \cap \xi^{k}(x) &= \xi^{n-k}(y) \cap \xi^{k}(x) = \{ 0\}.
    \end{align*}

    We conclude that any lifts $\overline{\phi}, \overline{\psi}: p \to \bbR^n$ of $\phi|_p$ and $\psi|_p$ have rank $1$.
    As $\text{Im}(\phi) \not\subset \xi^k(x)$ and $\text{Im}(\psi) \not\subset \xi^{n-k+1}(y)$ the maps $\{\phi\}$ and $\{\psi\}$ have nontrivial image and hence rank $1$.

    By the definitions of $\mathscr{L}_\rho^k$, $\mathscr{L}_{\rho}^{n-k+1}$, and $\{\cdot\}$, the image $\{\phi\}(p)$ is $(\xi^{n-k+1}(y) \cap \xi^{k+1}(x))/p$ and the image $\{\psi\}(p)$ is $(\xi^{k}(x) \cap \xi^{n-k+2}(y))/p$.
    As $\bbR^n/p = (\xi^k(x)/p) \oplus (\xi^{n-k+1}(y)/p)$, the images of $\{\phi\}$ and $\{\psi\}$ have nontrivial projections to distinct factors of the direct sum decomposition of $\bbR^n/p$ and so do not coincide.
    So $\{\phi\}$ and $\{\psi\}$ are linearly independent and $D_p \pitchfork$ has rank $2$, as desired.
    The description of the tangent space to $T_\rho^k$ at $p$ follows from our characterization of the images of $\{\phi\}$ and $\{\psi\}$ in the above.
    \end{proof}

\subsubsection{Product Structure} The corollary of the regularity of $T_\rho^k$ that we use later is an approximate-product metric model for compact subsets of $\Xi^2_\rho$.
To set notation, let $d_{\Xi^2}$ be the metric on $\Xi^2_\rho$ induced from its embedding in $\RP^{n-1}$.
Let $d_{\text{Gr}}$ be the product metric induced on $\xi^k(\partial \Gamma) \times \xi^{n-k+1}(\partial \Gamma)$ by the factors' embeddings in $\text{Gr}_k(\bbR^n)$ and $\text{Gr}_{n-k+1}(\bbR^n)$.

\begin{corollary}[Intersection Bilipschitz]\label{cor-bilipschitz}
    For any compact subset $K$ of $\Xi^2_k$, there is a constant $C_K > 0$ so that the map $(\pitchfork_\rho^k)^{-1}: K \to \xi^k(\partial \Gamma) \times \xi^{n-k+1}(\partial \Gamma)$ is $C_K$-bilipshitz onto its image with respect to the metrics $d_{\rm{Gr}}$ and $d_{\Xi}$.
\end{corollary}

\begin{proof}
    In the proof of Prop. \ref{prop-regular-off-frenet}, we saw $\pitchfork_\rho^k$ restricted to $\mathscr{D}_\rho^k$ is an injective $C^1$ immersion.
    So the restriction of $\pitchfork_\rho^k$ to any compact submanifold with boundary of $\mathscr{D}_\rho^k$ is a $C^1$ embedding.
    So $(\pitchfork_{\rho}^k)^{-1}$ is bilipschitz on any compact set in $\Xi^2_k$.
\end{proof}

\subsection{Foliations and Sliding Maps}\label{ss-foliations-sliding}

The tori $T_\rho^k$ have some further structure that we record here.
Namely, each torus $T_\rho^k$ has two canonical foliations by Frenet curves and there are natural bijections between the leaves of these foliations.
We also establish some regularity features of the natural bijections.
The rest of the paper is independent of this subsection.

\subsubsection{Definitions} The Frenet Restriction Lemma \ref{lemma-restriction} results in the two natural foliations of $\partial \Gamma^{2}$ mapping to foliations of $T_\rho^k$ by hyperconvex Frenet curves, which we record:
\begin{definition}
    Let $\mathcal{G}_{k}$ be the foliation of $T_\rho^k$ by hyperconvex Frenet curves in copies of $\RP^{k-1}$ with leaves given for $x \in \partial \Gamma$ by $g_x = \{ \Phi_\rho^k(x, y) \mid y \in \partial \Gamma \}$.
    Let $\mathcal{H}_{k}$ be the foliation of $T_\rho^k$ by in copies of $\RP^{n-k}$ with leaves given for $x \in \partial \Gamma$ by $h_x = \{ \Phi_\rho^k(y, x) \mid y \in \partial \Gamma \}$.   
\end{definition}
    
There are distinguished bijections between leaves of $\mathcal{G}_k$ (resp. $\mathcal{H}_k$), coming from the global product structure of $T_\rho^k$.
The case of $\mathcal{H}_2$ for $\text{PSL}_4(\bbR)$ played a role in \cite{nolte2023leaves}.

\begin{definition}[Sliding Maps]\label{def-sliding}
    Let $x_1, x_2 \in \partial \Gamma$. Define ${\rm{sl}}^{\mathcal{G}}_{x_1 \to x_2}: g_{x_1} \to g_{x_2}$ with $g_{x_1}, g_{x_2}$ leaves of $\mathcal{G}_k$ by:
    \begin{align*}
         {\rm{sl}}^{\mathcal{G}}_{x_1 \to x_2} : \xi^k(x_1) \cap \xi^{n+1-k}(y) &\mapsto \xi^k(x_2) \cap \xi^{n+1-k}(y) & y \neq x_1, x_2 \\
         \xi^k(x_1) \cap \xi^{n+1-k}(x_2) & \mapsto \xi^1(x_2) \\
         \xi^1(x_1) & \mapsto \xi^k(x_2) \cap \xi^{n+1-k}(x_1).
    \end{align*}
    Similarly, define ${\rm{sl}}^{\mathcal{H}}_{x_1 \to x_2}: h_{x_1} \to h_{x_2}$ with $h_{x_1}$, $h_{x_2}$ leaves of $\mathcal{H}_k$ by 
       \begin{align*}
         {\rm{sl}}^\mathcal{H}_{x_1 \to x_2} : \xi^k(y) \cap \xi^{n+1-k}(x_1) &\mapsto \xi^k(y) \cap \xi^{n+1-k}(x_2) & y \neq x_1, x_2 \\
         \xi^k(x_2) \cap \xi^{n+1-k}(x_1) & \mapsto \xi^1(x_2) \\
         \xi^1(x_1) & \mapsto \xi^k(x_1) \cap \xi^{n+1-k}(x_2).
    \end{align*}
\end{definition}

\subsubsection{Sliding Regularity} We next address the regularity of sliding maps.

\begin{lemma}
    For $x_1, x_2 \in \partial \Gamma$ and $k \in \{1, ..., k_n\}$ and $\mathcal{K} \in \{\mathcal{G}_k , \mathcal{H}_k\}$, the restriction of ${\rm{sl}}_{x_1 \to x_2}^{\mathcal{K}}$ with domain $k_{x_1}$ to $k_{x_1} - [(\Xi^1_\rho \cap k_{x_1}) \cup ({\rm{sl}}_{x_1 \to x_2}^{\mathcal{K}})^{-1}(\Xi^1_\rho)]$ is $C^1$ with nonzero derivative.
\end{lemma}

\begin{proof}
    We take $\mathcal{K} = \mathcal{G}_k$; the other case is analogous.

    Since any hyperconvex Frenet curve $\eta : \partial \Gamma \to \mathcal{F}(\bbR^n)$ has $C^1$ image in $\RP^{n-1}$, the Frenet Restriction Lemma implies that both $g_{x_1}$ and $g_{x_2}$ are embedded $C^1$ submanifolds of $\RP^{n-1}$.

    For $x \in \partial \Gamma$, Let $I_{x}^k : (\xi^{n-k+1}(\partial \Gamma) - \xi^{n-k+1}(x)) \to g_x^k$ be $\xi^{n-k+1}(y) \mapsto \xi^{n-k+1}(y) \cap \xi^k(x)$.
    The analysis of the proof of Prop. \ref{prop-regular-off-frenet} shows $I_x^k$ is $C^1$ with differential of full rank. As $\text{sl}^{\mathcal{K}}_{x_1 \to x_2}= I_{x_2}^k \circ (I_{x_1}^k)^{-1}$, the claim follows from the chain rule.
\end{proof}

\section{Bouquet Coupling}\label{s-bounds}

We prove Theorems C and D in this section.
Most of the discussion is on the bounds in Theorem C, with the upper bound the main point.
We prove the lower bound in \S \ref{ss-bouquet-lower-bounds} and the upper bound in \S \ref{ss-bouquet-upper-bounds}.
We address completeness in \S \ref{ss-metrics-completeness}.
A first observation is:
\begin{lemma}
    $D_B$ satisfies the triangle inequality, with possible value $\infty$.
\end{lemma}

\begin{proof}
    Optimal H\"older exponents are in $[0,1]$ and supermultiplicative under composition.
\end{proof}

\subsection{Lower Bounds}\label{ss-bouquet-lower-bounds}
We now prove the lower bounds for $D_B$ in Theorem C by examining explict sequences of points as in \S \ref{s-teich-space-case} (c.f. also \cite{benoist2004convexesI}).

The following is the general form of the upper bound on H\"older exponents that we prove and use as in \S \ref{s-teich-space-case}.
It is relevant to us because of the Frenet Restriction Lemma.

\begin{lemma}\label{lemma-reg-upper-bounds}
    Suppose $V$ is a real vector space of dimension $m$ and $g, h \in {\rm{PGL}}(V)$ are real-diagonalizable with $\lambda(g) = (\ell_1(g), ..., \ell_m(g))$ and $\lamdba(h) = (\ell_1(h), ..., \ell_m(h))$ each in $\mathring{\mathfrak{a}}^+$.
    Let $v_1(g), ..., v_m(g)$ and $v_1(h), ..., v_m(h)$ be eigenvectors corresponding to $\ell_1(g), ..., {\ell_m(g)}$ and $\ell_1(h), ... ,\ell_m(h)$, respectively.

    Suppose $M_1$ and $M_2$ are $C^1$ circles in $\bbP(V)$ so $[v_1(g)] \in M_1$ with ${\rm{T}}_{[v_1(g)]} M_1 = [v_1(g)] \oplus[v_2(g)]$ and $[v_1(h)] \in M_2$ with ${\rm{T}}_{[v_1(h)]} M_2 = [v_1(h)] \oplus [v_2(h)]$. Let $\phi: M_1 \to M_2$ be a homeomorphism so that $\phi([v_1(g)]) = [v_1(h)]$ and $\phi(gp) = h \phi(p)$ for all $p \in M_1$.

    Then $\phi$ is not $\alpha$-H\"older for any $\alpha > \ell^{\alpha_1}(h)/\ell^{\alpha_1}(g)$.
\end{lemma}

\begin{proof}
        Take affine charts $\mathcal{A}_{g}$ and $\mathcal{A}_h$ for $\bbP(V)$ so that in $\mathcal{A}_g$:
        \begin{enumerate}
            \item $[v_1(g)]$ is at the origin,
            \item $[\text{span}(v_2, ...,v_m)] $ is the hyperplane at infinity,
           \item For $j = 2, ...,m$, the $x_{j-1}$ axis is the intersection of $[\text{span}(v_1,v_j)]$ with $\mathcal{A}_g$.
        \end{enumerate}
        Arrange for the same conditions for $h$ to hold on $\mathcal{A}_h$.
        Outside of a neighborhood of the hyperplane at infinity, the metrics in $\mathcal{A}_g$ and $\mathcal{A}_h$ are multiplicatively comparable to a reference metric on $\bbP(V)$.
        Let $x_1^g$ and $x_1^h$ be the coordinate functions of $\mathcal{A}_g$ and $\mathcal{A}_h$.
        
        Consider the following sequence of points $(z_n, w_n) \in M_1$.
        Take $z_n = [v_1(g)]$ for all $n$.
        From the hypothesis on the tangent line to $M_1$ at $[v_1(g)]$ take $w_0$ so that there is an interval $I_g \subset M_1$ with endpoints $[v_1(g)]$ and $w_0$ so that $\frac{1}{2}|x_1^g(w)| \leq d_{\mathcal{A}_g}(0, w) \leq 2 |x_1^g(w)|$ for all $w \in I_g$, and the analogous claim holds for $h$ in $\mathcal{A}_h$.
        Take $w_n = g^n w_0$ for all $n > 0$.

        Now let $\alpha > \ell^{\alpha_1}(h)/\ell^{\alpha_1}(g)$.
        Then for an $n$-independent constant $C$,
        \begin{align*}
            \frac{d_{\mathcal{A}_h}(\phi(z_n), \phi(w_n))}{d_{\mathcal{A}_g}(z_n, w_n)^\alpha} &= \frac{d_{\mathcal{A}_h}(0,h^n \phi(w_0))}{d_{\mathcal{A}_g}(0,g^nw_0)^\alpha} \geq  \frac{|x_1^h(h^n \phi(w_0))|}{4|x_1^g (g^nw_0)|^\alpha} = \frac{e^{-n \ell^{\alpha_{1}}(h)}|x_1^h(\phi(w_0))|}{4e^{-\alpha n\ell^{\alpha_{1}}(g)}|x_1^g(w_0)|^\alpha} =  \frac{e^{n\alpha \ell^{\alpha_{1}}(g)}}{e^{\ell^{\alpha_1}(h)}} C,
        \end{align*} which grows arbitrarily large.
     So $\phi$ is not $\alpha$-H\"older.
\end{proof}

We now establish our lower bounds for $D_B$. Let us set notation.
For Hitchin representations $\rho_a, \rho_b: \Gamma \to \text{PSL}_n(\bbR)$ and $\gamma \in \Gamma - \{e\}$, define $r_{ab}^k(\gamma) = \ell^{\alpha_k}_{\rho_a}(\gamma)/\ell^{\alpha_k}_{\rho_b}(\gamma)$, and let $h_{ab}^k = \inf_{\gamma \in \Gamma - \{e\}} r_{ab}^k(\gamma)$. A useful remark is that $h_{ab}^k = h_{ab}^{n-1-k}$ because $\ell_{\rho}^{\alpha_k}(\gamma) = \ell_{\rho}^{\alpha_{n-1-k}}(\gamma^{-1})$.

\begin{proposition}
    For $n > 3$ and $\rho_a, \rho_b \in {\rm{Hit}}_n(S)$, ${\rm{max}}_{i=1,...,n-1}d_{\rm{Th}}^{\alpha_k}(\rho_a, \rho_b) \leq D_B(\rho_a, \rho_b)$.
\end{proposition}
    
\begin{proof}
    Denote by $\alpha_{ab}$ be the supremum of the $\alpha > 0$ so that $\Phi_{ab}$ is $\alpha$-H\"older.

    Let $\gamma \in \Gamma - \{e\}$.
    Lemma \ref{lemma-reg-upper-bounds} applied to $\xi^1_{\rho_b}(\partial \Gamma)$ and $\xi^1_{\rho_a}(\partial \Gamma)$ at the points $\xi^1_{\rho_b}(\gamma^+)$ and $\xi^1_{\rho_a}(\gamma^+)$ with the map $\Phi_{ab}$ shows that $\alpha_{ab} \leq r_{ab}^1(\gamma)$, so that $\alpha_{ab} \leq h_{ab}^1(\gamma)$.

    Next fix $k \in 2, ..., k_n$.
    Consider the curves $\eta_{\gamma^-, n-k+1}$, and $\nu_{\gamma^-, n-k+1}$ induced from the Frenet Restriction Lemma by $\xi_{\rho_b}$ and $\xi_{\rho_a}$, respectively.
    Let $M_b = \eta_{\gamma^-, n-k+1}(\partial \Gamma)$ and $M_a = \nu_{\gamma^-, n-k+1}(\partial \Gamma)$.
    Then $\rho_{c}(\gamma)|_{\xi^{n-k+1}(\gamma^-)}$ has eigenvalues $v_k(\rho_c(\gamma)), v_{k+1}(\rho_c(\gamma)), ..., v_n(\rho_c(\gamma))$ for $c \in \{a,b\}$.
    Additionally, $\eta^1_{\gamma^-, n-k+1}(x) = \Phi_{\rho_b}^k(x, \gamma^-)$ and $\nu^1_{\gamma^-, n-k+1}(x) = \Phi_{\rho_a}^k(x, \gamma^-)$ for $x \in \partial \Gamma$.
    We focus on the points $\eta^1_{\gamma^-, n-k+1}(\gamma^+) = [v_k(\rho_b(\gamma))]$ and $\nu^1_{\gamma^-, n-k+1}(\gamma^+) = [v_k(\rho_a(\gamma))]$.    
    
    Since any hyperconvex Frenet curve $\xi = (\xi^1, \xi^2, ..., \xi^l)$ has $C^1$ image of $\xi^1$ in projective space with tangent line $\xi^2(x)$ at $\xi^1(x)$ for all $x \in \partial\Gamma$, the tangent lines to $M_c$ at $[v_k(\rho_c(\gamma))]$ are $[v_{k}(\rho_c(\gamma))] \oplus [ v_{k+1}(\rho_c(\gamma))]$ for $c \in \{a,b\}$.
    Lemma \ref{lemma-reg-upper-bounds} applies to the restriction of $\Phi_{ab}$ to these curves, and produces an upper bound $\alpha_{ab} \leq r_{ab}^k$, so that $\alpha_{ab} \leq h_{ab}^k$.

    So $\max_{k=1,...,k_n} d_{\text{Th}}^{\alpha_k}(\rho_a, \rho_b) \leq D_B(\rho_a, \rho_b)$, and the equality $h_{ab}^k = h_{ab}^{n-1-k}$ gives the claim.
\end{proof}

\subsection{Upper Bounds}\label{ss-bouquet-upper-bounds}
What remains to be shown in Theorem C is the upper bound on $D_B$.
The strategy roughly parallels the computation of lower bounds for H\"older constants in \S \ref{s-teich-space-case}, with adaptations to the technical challenges of our setting.
In particular, we split our analysis for points close to and far away from the image of $\xi^1$.
See \S \ref{ss-outline}.

\subsubsection{Estimates Away from the Frenet Curve}
Away from $\Xi^1_\rho$, the regularity of $\Xi^2_\rho$ allows for a reduction to the regularity of the curves $\xi^k(\partial \Gamma)$ in Grassmannians.
Put \begin{align}
        \alpha_{ab}^L = \min_{k=1, ..., n-1} \left( \inf_{\gamma \in \Gamma - \{e\}} \frac{\ell_{\rho_a}^{\alpha_k}(\gamma)}{\ell_{\rho_b}^{\alpha_k}(\gamma)}\right) 
    \end{align}

\begin{lemma}\label{lemma-holder-regularity-off-frenet}
    Let $\rho_a, \rho_b \in {\rm{Hit}}_n(S)$ and $K$ be a compact subset of $\Xi^2_{\rho_b}$.
    Then the restriction of $\Phi_{ab}$ to $K$ is $\alpha_{ab}^L$-H\"older.
\end{lemma}

\begin{proof}
From Lemma \ref{lemma-boquet-param} it follows that $K$ is a disjoint union $K_1 \sqcup ... \sqcup K_{k_n}$ of compact subsets of $\Xi_k^2$.
It suffices to prove the result for the individual compact sets $K_k$.

So let $k \in \{2, ..., k_n\}$ be given.
Denote by $\phi_{ab}^k \times \phi_{ab}^{n-k+1}$ the product map $\xi^{k}_{\rho_b}(\partial \Gamma) \times \xi^{n-k+1}_{\rho_b}(\partial \Gamma) \to\xi^{k}_{\rho_a}(\partial \Gamma) \times \xi^{n-k+1}_{\rho_a}(\partial \Gamma) $ induced by $\phi_{ab}^{k}$ and $\phi_{ab}^{n-k+1}$. 
Note $\phi_{ab}^k \times \phi_{ab}^{n-k+1}$ is $\alpha_{ab}^L$-H\"older because the optimal H\"older exponent of a product map is the minimum of the optimal H\"older exponents of its factors.
Also, by definition, $\Phi_{ab} = \pitchfork_{\rho_a}^k \circ (\phi_{ab}^k \times \phi_{ab}^{n-k+1}) \circ \pitchfork_{\rho_b}^{-1}$ on $K_k$.
Since H\"older regularity is invariant under bilipschitz mappings, Cor. \ref{cor-bilipschitz} here and Theorem 1.9 of \cite{tsouvalas2023holder} show that $\Phi_{ab}$ restricted to $K_k$ is $\alpha_{ab}^L$-H\"older.
\end{proof}

\begin{remark}
    Giving a direct proof of Lemma \ref{lemma-holder-regularity-off-frenet} with the standard approach in the literature for similar estimates (e.g. \cite{guichard2005regularity}, \cite{tsouvalas2023holder}, the below) seems quite challenging because of the presence of fixed-points of $\rho_a(\gamma)$ and $\rho_b(\gamma)$ $(\gamma \in \Gamma)$ of intermediate-magnitude eigenvalues in $\Xi^2_\rho$.
    In particular, the regularity of $\Xi^2_\rho$ appears essential to the tractability of the above estimate.
\end{remark}

\subsubsection{Loxodromy}\label{ss-quant-loxodromy}
In this paragraph, we recall the relevant improvement to proximality to our setting and some useful lemmas from work of Guichard \cite{guichard2005regularity}.

Loxodromy is an improvement to proximality that includes all eigenvalue gaps.
We say that $g \in \text{GL}_n(\bbR)$ is \textit{loxodromic} if $g$ is diagonalizable over $\mathbb{R}$ with eigenvalues $\lambda_1, ..., \lambda_n$ of $g$ listed with multiplicity, and $|\lambda_i|, |\lambda_j|$ are distinct for all $i \neq j$.
Hitchin representations $\rho \in \text{Hit}_n(S)$ have the property that $\rho(\gamma)$ is loxodromic for all $\gamma \in \Gamma -\{e\}$ \cite{labourie2004anosov}.

We adopt the convention that $\lambda_1, ...,\lambda_n$ are ordered by decreasing modulus, and the corresponding vectors are labelled $v_1, ..., v_n$ and are taken to have norm $1$ with respect to a background inner product on $\bbR^n$.
We also write $\ell_i = \log |\lambda_i|$ for $i=1,...,n$.
If $g$ is loxodromic, then for all $k = 1,..., n-1$ the $k$-th exterior power $\Lambda^k g$ acts proximally on $\mathbb{P}(\Lambda^k V)$.
For $g$ loxodromic, the map $\lambda: g \mapsto (\ell_1(g), ..., \ell_n(g))$ coincides with the Jordan projection $\lambda$.

Let $\delta_0$ be a metric induced on $\bbP(V)$ by an inner product on $V$, from realizing $\bbP(V)$ as the quotient of the unit sphere in $V$ by negation.
Then the inner product inducing $\delta_0$ induces an inner product on $\Lambda^k V$ and hence metrics $\delta_0^k$ on $\bbP(\Lambda^kV)$ for $k=1,...,n-1$.
If $g$ is loxodromic, with eigenspaces $[v_1], [v_2], ..., [v_n]$, taking norm-$1$ representatives $v_i$ with respect to the background inner product, then specifying a new inner product in which $\{v_i\}_{1,..., n}$ is orthonormal produces new metrics $\delta_g$ on $\bbP(V)$ and $\delta_g^k$ on $\bbP(\Lambda^kV)$ for all $k = 1, ..., n-1$.

We say $g \in \text{GL}(V)$ is \textit{$(r,\varepsilon)$-loxodromic} if $\Lambda^k g$ is $(r,\varepsilon)$-proximal with respect to $\delta_0^k$ for all $k = 1,..., n-1$.
If, furthermore, $g^{-1}$ is $(r,\varepsilon)$-loxodromic, we say that $g$ is \textit{$(r,\varepsilon)$-biloxodromic}.
The following is a standard consequence of Abels-Margulis-Soifer's work \cite{abelsMargulisSoifer1995semigroups}:

\begin{fact}[Corollary of \cite{abelsMargulisSoifer1995semigroups}]\label{fact-loxodromy}
    Let $\rho_a, \rho_b$ be Hitchin, let $r > 0$ be sufficiently small, and let $r > \varepsilon > 0$. Then there is a finite set $F \subset \Gamma$ so that for all $\gamma \in \Gamma$ there is some $\eta \in F$ so that $\rho_a(\eta \gamma)$ and $\rho_b(\eta \gamma)$ are $(r, \varepsilon)$-loxodromic.
\end{fact}

\subsubsection{Loxodromy Refinement}\label{ss-loxodromy-refinement} 
Our first task in analyzing $\Phi_{ab}$ near $\Xi^1$ is to prove the following general analogue of the Proximality Refinement Claim \ref{claim-teich-better-biproximality}.
We adopt the notation that $B_n^{(2)} = \{ (p,q) \in B_n^2 \mid p \neq q \}$.

\begin{proposition}[Loxodromy Refinement]\label{prop-loxodromy-refinement}
    Let $\rho: \Gamma \to {\rm{PSL}}_n(\bbR)$ ($n\geq 2$) be Hitchin, $r > 0$ be sufficiently small, and $r > \varepsilon > 0$. Let $\Gamma^{\rho}_{(r,\varepsilon)}$ be the collection of $\gamma \in \Gamma$ so that $\rho(\gamma)$ is $(r,\varepsilon)$-loxodromic, and
    let $K \subset B_n^{(2)}$ be compact.
    
    Then for any $\varepsilon' > 0$ there is a $\tau(\varepsilon') > 0$ so that for all $\gamma \in \Gamma^\rho_{(r,\varepsilon)}$ if there exists $(p, q) \in B_n^{(2)}$ with both $\delta_0(\Phi_{\rho}(p), \Phi_\rho(q)) < \tau(\varepsilon')$ and $(\gamma p, \gamma q) \in K$ then $\rho(\gamma)$ is $(r, \varepsilon')$-biloxodromic.
\end{proposition} 

We begin with preliminaries. We shall repeatedly use the following:
\begin{proposition}[Guichard \cite{guichard2005regularity} Prop. 18]\label{prop-guichard-compact}
    Let $r > 0$ be sufficiently small. There is a compact subset $K_r \subset {\rm{GL}}_n(\bbR)$ so that for any $\varepsilon \in (0,r)$ and $(r,\varepsilon)$-loxodromic $g \in {\rm{GL}}_n(\bbR)$ with norm-$1$ representatives $v_1, ..., v_n$ of the eigenslines of $g$, and with $e_1, ..., e_n$ a fixed orthonormal basis of the background inner product, there is a $T(g) \in K_r$ so that $(v_1, ..., v_n)^T = T(g) (e_1, ..., e_n)^T$.
\end{proposition}

The below is a corollary of the preceding (\cite{guichard2005regularity} Cor. 19) and the observation that inducing metrics on exterior products is natural with respect to change of basis.

\begin{corollary}[Guichard \cite{guichard2005regularity}]\label{cor-comparability}
    For any $r> 0$ there is a $C_r > 0$ so for all $\varepsilon \in (0,r)$, $k \in \{1,...,n-1\}$, and any $(r,\varepsilon)$-proximal $g$, the metrics $\delta_g^k$ and $\delta_0^k$ are $C_{r}$-multiplicatively comparable.
    That is, $C_{r}^{-1}\delta_g^k(x,y) \leq \delta_0^k(x,y) \leq C_{r} \delta_{g}^k(x,y)$ for all $x,y \in \bbP(\Lambda^k\bbR^n)$.
\end{corollary}

Prop. \ref{prop-loxodromy-refinement} uses the following link between eigenvalue gaps and quantitative proximality.

\begin{lemma}\label{lemma-better-proximality-from-eigenvalues}
    Fix a closed cone $L$ contained in $\mathring{\mathfrak{a}}^+$ and $r > \varepsilon > 0$ sufficiently small.
    For any $\varepsilon' > 0$ there is a $T > 0$ so that if $g \in {\rm{PSL}}_n(\bbR)$ is $(r,\varepsilon)$-loxodromic and satisfies:
    \begin{enumerate}
        \item {\rm{(Jordan Projection Control)}} $\lambda(g) \in L$,
        \item {\rm{(Hilbert Length Control)}} ${\rm{H}}(g) > T$,
    \end{enumerate} then $g$ is $(r,\varepsilon')$-biloxodromic.
\end{lemma}

\begin{proof}
    We may replace $L$ with $L \cup \iota(L)$, with $\iota$ the opposition involution.
    As $\lambda(g^{-1}) = \iota (\lambda(g))$ and $H(\lambda(g^{-1})) = H(\lambda(g))$, it suffices to prove $(r,\varepsilon')$-loxodromy instead of biloxodromy. 
    
    Let us begin by noting that since $L$ is a closed cone in $\mathring{\mathfrak{a}}^+$ there is a constant $c_L > 0$ so $\alpha_k(x) > c_L \text{H}(x)$ for all $x \in L$ and $k \in \{ 1, ..., n-1\}$. So it suffices to take $\ell^{\alpha_k}(g) > T$ for all $k \in \{1 ,..., n-1\}$ in our analysis. As the desired property is invariant under replacing the metrics $\delta_0^k$ on $\bbP(\Lambda^k \bbR^n)$ with uniformly multiplicatively comparable metrics, we may replace $\delta_0^k$ on $\bbP(\Lambda^k\bbR^n)$ by $\delta_g^k$ for all $k \in \{1, ...,n-1\}$.
    
    We now compute. Let $k \in \{1, ..., n-1\}$ be fixed.
    Let $[v_1^k], ..., [v_N^k]$ ($N = \binom{n}{k}$) be the eigenlines of $\Lambda^k g$, ordered with nondecreasing modulus of eigenvalues, and with corresponding eigenvalues $\lambda_1^k, ..., \lambda_{N}^k$.
    The largest-modulus eigenvalue eigenspace $[v_1^k]$ of $\Lambda^kg$ is $[v_1 \wedge ... \wedge v_k]$ with corresponding eigenvalue $\lambda^k_1 = \lambda_1 ... \lambda_k$.
    Note $|\lambda^k_1/\lambda^k_j| \geq |\lambda_k/\lambda_{k+1}| > 1$ for all $j > 1$.
    
    Let $\mathcal{A}$ be an affine chart with $X_{\Lambda^kg}^<$ the plane at infinity, $x_{\Lambda^k g}^+$ the origin, and the $x_j$-axis ($j=1,..., N -1$) given by the intersection of the projective line $x_{g}^+ \oplus \bbR {v_{j+1}^k}$ for $1 \leq j \leq N-1$, and with a point that is uniformly separated from all hyperplanes given by the span of $N-1$ eigenlines of $\Lambda^kg$ placed at $(1, ..., 1)$.
    
    Fix $\varepsilon' > 0$.
    Then there are $C_{\varepsilon'}, R_{\varepsilon'}, r_{\varepsilon'} > 0$ so the Euclidean metric $d_{\mathcal{A}}$ on $\mathcal{A}$ satisfies:
    \begin{enumerate}
        \item The restriction of $d_{\mathcal{A}}$ to $\bbP(\Lambda^k\bbR^n) - N_{\varepsilon'}(X_{\Lambda^kg}^<)$ has $C_{\varepsilon'}^{-1} \delta_g^k \leq d_{\mathcal{A}} \leq C_{\varepsilon'} \delta_g^k$,
        \item $\bbP(\Lambda^k \bbR^n) - N_{\varepsilon'}(X_{\Lambda^kg}^<)$ is contained in $N_{R_\varepsilon'}(0)$ in $\mathcal{A}$ with respect to $d_\mathcal{A}$,
        \item $N_{\varepsilon'}(x_{\Lambda^kg}^+)$ contains $N_{r_{\varepsilon'}}(0)$ in $\mathcal{A}$ with respect to $d_\mathcal{A}$.
    \end{enumerate} 
    Now, $\Lambda^k g$ acts on $\mathcal{A}$ by $$g[x_1:x_2:...:x_{n-1}:1] = \left[(\lambda_2^k/\lambda_1^k) x_1 : (\lambda_3^k/\lambda_1^k)x_2 : ... : \left(\lambda_{N}^k/\lamdba_1^k\right)x_{{N} -1}:1\right].$$
    This map is Lipschitz with Lipschitz constant $\lambda_2^k / \lambda_1^k = \exp(-\alpha_k(g))$, and maps $N_R(0)$ to $N_{\lambda_2^k R/ \lambda_1^k}(0)$ for all $R > 0$. The claim follows.
\end{proof}

We now prove Prop. \ref{prop-loxodromy-refinement}.
The point is that in the model metric $\delta_g$ the differential of $g$ can be controlled in terms of the Hilbert length.

\begin{proof}[Proof of Prop. \ref{prop-loxodromy-refinement}]
    Let $n, K, \rho, r, \varepsilon,$ and $\varepsilon'$ be given as stated.
    Since limit cones of Hitchin representations are contained $\mathring{\mathfrak{a}}^+$ (\cite{potrie2017eigenvalues-entropy} Prop. 4.4), we may apply Lemma \ref{lemma-better-proximality-from-eigenvalues} with $\mathcal{L}_{\rho(\gamma)}$.
    In particular, Lemma \ref{lemma-better-proximality-from-eigenvalues} shows there is a $T < \infty$ so that it suffices to produce a $\tau$ so that if $\delta_0(\Phi_\rho ( p), \Phi_\rho(q)) < \tau$ and $(\gamma p , \gamma q) \in K$ with $\gamma \in \Gamma_{(r,\varepsilon)}^\rho$ then $\ell^{\text{H}}_\rho(\gamma) > T$.
    By Corollary \ref{cor-comparability}, the identity maps between the metrics $\delta_{\rho(\gamma)}$ and $\delta_0$ for all $\gamma$ under consideration are uniformly bilipschitz. So we are free to work with the metrics $\delta_{\rho(\gamma)}$ on $\RP^{n-1}$ throughout the following.
    
    \begin{claim}\label{claim-quantitative-bilipschitz}
        Let $r > 0$ be sufficiently small. For every $T > 0$ there is a $B = B(T,r, \rho) < \infty$ so if $\gamma \in \Gamma^{\rho}_{r,\varepsilon}$ for any $\varepsilon < r$ and $\ell^{\rm{H}}_{\rho}(\gamma) \leq T$ then the action of $\rho(\gamma)$ on $(\RP^{n-1}, \delta_{\rho(\gamma)})$ is $B$-bilipschitz.
    \end{claim}

    \begin{proof}
    Since we work with $\delta_{\rho(\gamma)}$, we may treat $\rho(\gamma)$ as diagonal with respect to a fixed basis of $\bbR^n$.
    The collection $\mathscr{C}_{\rho,T}$ of elements $h$ of $\text{PGL}_n(\bbR)$ that are diagonal with respect to this basis, have $\lambda(h) \in \mathcal{L}_\rho$, and have $\ell^{\rm{H}}(h) \leq T$ is compact.
    Since every $h$ in the compact family $\mathscr{C}_{\rho,T}$ acts by diffeomorphisms on $\RP^{n-1}$, there are uniform lower and upper bounds on $D_ph(v)/||v||$ for $p \in \RP^{n-1}$, nonzero $v \in \text{T}_p\RP^{n-1}$ and $h \in \mathscr{C}_{\rho,T}$.
    \end{proof}

    Now let $T > 0$.
    Since $K$ is compact, there is a $\nu > 0$ so the lower bound $\delta(\Phi_\rho(p), \Phi_\rho(q)) > \nu$ holds for all $(p,q) \in K$.
    We claim with $B(\rho,T)$ as supplied by Claim \ref{claim-quantitative-bilipschitz}, $\tau = \nu/B(\rho,T)$ satisfies the condition we reduced to above.
    So suppose there is a pair $(p,q) \in B_n^{(2)}$ so that $\delta_g(\Phi_\rho(p), \Phi_\rho(q)) < \nu/B(\rho,T)$ and $(\gamma p , \gamma q) \in K$. Then $\delta_g(\rho(\gamma) \Phi_\rho(p), \rho(\gamma) \Phi_\rho(q)) > \nu$, so the action of $\rho(\gamma)$ on $\RP^{n-1}$ is not $B(\rho,T)$-Lipschitz and Claim \ref{claim-quantitative-bilipschitz} shows $\ell^{\text{H}}_\rho(\gamma) > T$.
\end{proof}

\subsubsection{Estimates Near the Frenet Curve}

We now prove our lower bound for the H\"older exponent of $\Phi_{ab}$.
Let \begin{align}
    \alpha_{ab}^U = \min\left(\alpha_{ab}^L, \inf_{\gamma \in \Gamma - \{e\}} \frac{\ell_{\rho_a}^{\alpha_1}(\gamma) }{ \ell_{\rho_b}^{\rm{H}}(\gamma)} \right)\label{eq-holder-lb}
\end{align}
Note that $\alpha_{ab}^U > 0$ as the limit cones $\mathcal{L}_{\rho_a}$ and $\mathcal{L}_{\rho_b}$ are contained in $\mathring{\mathfrak{a}}^+$ and $d_{\text{Th}}^{\alpha_1}(\rho_b, \rho_a)$ is finite.

\begin{proposition}
    Let $\rho_a, \rho_b: \Gamma \to {\rm{PSL}}_n(\bbR)$ ($n > 3$) be Hitchin.
    Then $\Phi_{ab}$ is $\alpha_{ab}^U$-H\"older.
\end{proposition}

\begin{proof}
    The idea is that points that are both close together in $\mathcal{B}_{\rho_c}$ and close to $\Xi^1_{\rho_c}$ ($c \in \{a,b\}$) are the image under $\Phi_{\rho_c}$ of pairs of points in $B_n$ specified by a quadruple of clustered-together points in $\partial \Gamma$.
    The argument of the Fuchsian case is effective for such points with the two modifications of using loxodromy in place of proximality and the relevant application of the uniform convergence group property of $\Gamma$ using the points on the outside of the cluster.
    The need to control more quantities in this setting makes the proof rather technically involved; the reader may wish to refer back to \S \ref{s-teich-space-case} for a model.

    Our proof breaks into the following stages, which we have labelled below:
    \begin{enumerate}
        \item \label{step-1-notation-clustering} Notation and basic features of pairs of points under consideration,
        \item \label{step-2-proximality-construction} Initial Abels-Margulis-Soifer \cite{abelsMargulisSoifer1995semigroups} application and the main construction,
        \item \label{step-3-loxodromy-refinement} Application of loxodromy refinement to control placements of involved points, 
        \item \label{step-4-affine-charts} Uniformly controlled choice of affine chart,
        \item \label{step-5-actual-computation} The final computation.
    \end{enumerate}

    \medskip
    \noindent (\ref{step-1-notation-clustering}). Let us set notation.
    Fix a background metric $d_{\partial \Gamma}$ on $\partial \Gamma$.
    Writing $B_n= (\partial \Gamma_1^2 \sqcup ... \sqcup \partial \Gamma_{k_n}^2)/\sim$, we distinguish between membership in each term by denoting a pair $(x,y) \in \partial \Gamma_j^2$ ($j \in \mathcal{K}_n$) by $(x,y)_j$. For any $\kappa, \eta > 0$ and $c \in \{a,b\}$, we define $\Xi^2_{c}(\kappa) \subset \Xi^2_{\rho_c}$ to be the complement of $N_\kappa(\Xi^1_{\rho_c})$, and $\text{Far}_c(\eta) \subset \mathcal{B}_{\rho_c}^{(2)}$ to be the points $(l_1, l_2)$ so $\delta_0(l_1, l_2) \geq \eta$.
    Let $U_{\kappa,c}$ be $N_{\kappa}(\Xi^1_{\rho_c})$ and $V_{\eta, c} = \mathcal{B}_{\rho_c}^{(2)} - \text{Far}_{c}(\eta)$.
    Let $\mathscr{O}_{\kappa} \subset B_{n}$ be $\Phi_a^{-1}(U_{\kappa, a}) \cap \Phi_b^{-1}(U_{\kappa, b})$.
    Let $\mathscr{Q}_{\kappa, \eta} \subset B_n^{(2)}$ be the collection of pairs $(p,q)$ with $p, q \in \mathscr{O}_{\kappa}$ and $\Phi_c(p, q) \in V_{\eta,c}$ for all $c \in \{a,b\}$.
    
    Using Lemma \ref{lemma-holder-regularity-off-frenet}, for any $\kappa$ and $\eta$, there are constants $C_{\Xi}(\kappa)$ and $C_{\text{Far}}(\eta)$ so that
    \begin{align}
        \frac{\delta_0(\Phi_{\rho_a}(p), \Phi_{\rho_a}(q))}{\delta_0(\Phi_{\rho_b}(p),\Phi_{\rho_b}( q))^{\alpha_{ab}^U}} &\leq C_{\Xi}(\kappa)  &(p, q \notin \mathscr{O}_{\kappa}),\label{eq-holder-lb-pre-controlled} \\
        \frac{\delta_0(\Phi_{\rho_a}(p), \Phi_{\rho_a}(q))}{\delta_0(\Phi_{\rho_b}(p),\Phi_{\rho_b}( q))^{\alpha_{ab}^U}} &\leq C_{\text{Far}}(\eta) &((\Phi_{\rho_b}(p),\Phi_{\rho_b}(q)) \in \text{Far}_b(\eta)). \nonumber 
    \end{align}
    We say points $x_1, x_2, x_3, x_4 \in \partial \Gamma$ are \textit{$\nu$-clustered} if $\text{max}_{1 \leq i \neq j \leq 4}d_{\partial \Gamma}(x_i, x_j) < \nu.$    
    We do not require the points be distinct.
    If $x_1,x_2, x_3, x_4$ are $\nu$-clustered, we call $\{x_1, x_2, x_3, x_3\}$ a $\nu$\textit{-cluster}.
    For any $\nu$ sufficiently small, the complement of any $\nu$-clustered $x_1, x_2, x_3, x_4 \in \partial \Gamma$ contains an interval $I$ with diameter exceeding $2\text{diam}(\partial \Gamma) / 3$.
    Then there are $x_l, x_r \in \{x_1, x_2,x_3,x_4\}$ so that $\partial I = \{x_l,x_r\}$ and for any $z \in I$ the triple $(x_l, x_r, z)$ is positively oriented.
    Call $x_l, x_r$ the \textit{edge points} of $x_1, x_2,x_3,x_4$.
    
    We shall repeatedly take the cluster associated to a pair of points $(p, q) \in \mathcal{Q}_{\eta, \kappa}$ by writing $p = (x_1, y_1)_i, q = (x_2, y_2)_j$ and considering $\{x_1, y_1, x_2,y_2\}$.
    Let us denote this cluster by $\text{cl}(p,q)$.
    We note that some data is forgotten by taking $\text{cl}(p,q)$: which copy of $\partial \Gamma^{2}_k$ points are in, and the order of points in $\partial \Gamma^{2}_k$.
    This data is salient away from $\Xi^1_{\rho_c}$ ($c \in \{a,b\}$), and does not end up needing to be explicitly used in the following.
    
    The following is a direct consequence of $\Phi_{\rho}$ being a homeomorphism:    
    \begin{claim}[Clustering Characterization]\label{claim-cluster-characterization}
    For any $\nu > 0$ there are $\eta(\nu), \kappa(\nu) > 0$ so that ${\rm{cl}}(p,q)$ is a $\nu$-cluster for all $(p, q) \in \mathscr{Q}_{\kappa, \eta}$.
    
    On the other hand, given any $\nu > 0$, there are strictly positive $\kappa(\nu)$ and $\eta(\nu)$ so that if $(p, q) \in B_n^{(2)}$ has ${\rm{cl}}(p, q)$ not a $\nu$-cluster, then for each $c \in \{a,b\}$ one of the following holds:
    
    \begin{enumerate}
        \item \label{frenet-separated} Both $\Phi_{\rho_c}(p), \Phi_{\rho_c}(q)$ are in $\Xi^2_c(\kappa(\nu))$,
        \item $(\Phi_{\rho_c}(p), \Phi_{\rho_c}(q)) \in {\rm{Far}}_c(\eta(\nu))$.
    \end{enumerate} 
    \end{claim}

    \begin{proof}
        The first claim is clear.
        The second claim is clear, with the weakened conclusion that one of $\Phi_{\rho_c}(p), \Phi_{\rho_c}(q) \in \Xi^2_c(\kappa'(\nu))$ or $(\Phi_{\rho_c}(p), \Phi_{\rho_c}(q)) \in \text{Far}_c(\eta'(\nu))$.
        We can then take $\eta'' > 0$ so that if $\Phi_{\rho_c}(p) \in \Xi^2_c(\kappa'(\nu))$ and $\Phi_{\rho_c}(q) \in U_{\kappa'(\nu)/2,c}$ then $(\Phi_{\rho_c}(p),\Phi_{\rho_c}( q)) \in \text{Far}_c(\eta'')$.
        Adjusting to $\kappa = \kappa'(\nu)/2$ and $\eta = \text{min}(\eta', \eta'')$ then gives the claim.
    \end{proof}

    Due to Lemma \ref{lemma-holder-regularity-off-frenet}, equations (\ref{eq-holder-lb-pre-controlled}), and Claim \ref{claim-cluster-characterization}, for the conclusion we may take arbitrarily small but fixed $\nu, \kappa, \eta > 0$ to be determined later and only consider points $(p, q) \in \mathscr{Q}_{\kappa, \eta}$ so that $\text{cl}(p, q)$ is a $\nu$-cluster.
    We call points $(p, q)$ satisfying the above \textit{$(\kappa, \eta, \nu)$-adapted}, and write the collection of $(\kappa,\eta, \nu)$-adapted pairs as $\mathscr{R}_{\kappa, \eta, \nu}$.

\medskip

\noindent (\ref{step-2-proximality-construction}). The next piece of setup parallels \S \ref{s-teich-space-case} directly. 
    Let $r > \varepsilon > 0$ be sufficiently small and fixed.
    By Fact \ref{fact-loxodromy}, there is a finite set $F \subset \Gamma$ so that for all $\gamma \in \Gamma$ there is an $\eta \in F$ so that $\rho_a(\eta \gamma)$ and $\rho_b(\eta \gamma)$ are $(r,\varepsilon)$-loxodromic.
    As $\Gamma$ is a uniform convergence group, there is a compact subset $K \subset \partial \Gamma^{(3)+}$ so that for any $t \in \partial \Gamma^{(3)+}$ there is a $\gamma \in \Gamma$ so that $\gamma t \in K$.
    Replacing $K$ with the compact set $L = FK$ and setting $$\Gamma_{(r,\varepsilon)}^{a,b} = \{ \gamma \in \Gamma \mid \rho_a(\gamma) \text{ and } \rho_b(\gamma) \text{ are } (r,\varepsilon) \text{ loxodromic}\},$$ for every $t \in \partial \Gamma^{(3)+}$ there is a $\gamma \in \Gamma_{(r,\varepsilon)}^{a,b}$ so $\gamma t \in L$.

    We analyze the following construction for the rest of the proof.    
    Let $(p, q) \in \mathscr{R}_{\kappa_0, \eta_0, \nu_0}$ for initial $ \kappa_0, \eta_0, \nu_0 > 0$ with $\nu_0$ less than $\text{diam}(\partial \Gamma)/3$. 
    Then $\text{cl}(p, q)$ is a $\nu_0$-cluster with at least two distinct points.
    Let $x_l \neq x_r$ be the edge points.
    Pick any third point $z$ so $d_{\partial \Gamma}(z, x) > \text{diam}(\partial \Gamma)/3$ for all $x \in \text{cl}(p, q)$.
    Then there is some $\gamma_{pq} \in \Gamma_{r, \varepsilon}^{a,b}$ so that $\gamma_{pq}(x_l, x_r, z) \in L$.

\medskip
\noindent (\ref{step-3-loxodromy-refinement}). As in \S \ref{s-teich-space-case}, our freedom to choose $\eta, \kappa,\nu$ small allows us to only consider $\gamma_{pq}$ so that both $\gamma_{pq}$ and $\gamma_{pq}^{-1}$ are elements of $\Gamma^{a,b}_{(r,\varepsilon')}$ for a fixed $\varepsilon' > 0$ of our choosing:

    \begin{claim}\label{claim-needed-proximality-improvement}
        For any $\varepsilon' > 0$ there are $\nu, \kappa, \eta > 0$ so that if $(p, q) \in \mathscr{R}_{\kappa, \eta, \nu}$ then $\gamma_{pq} \in \Gamma^{a,b}_{(r, \varepsilon')}$ and $\gamma_{pq}^{-1} \in \Gamma^{a,b}_{(r,\varepsilon')}$.
        We call such points {\rm{$(\kappa, \eta,\nu, \varepsilon')$-adapted}}, and write the collection of $(\kappa, \eta, \nu, \varepsilon')$-adapted pairs as $\mathscr{S}_{\kappa, \eta,\nu, \varepsilon'}$.
    \end{claim}

    \begin{proof}
        We retain the above notation.
        As $L$ is compact, there is a $r_1 > 0$ so that $d_{\partial \Gamma}(x_i, x_j) \geq r_1$ for all $(x_1, x_2, x_3) \in L$ and $1 \leq i \neq j \leq 3$. 
        As $\gamma_{pq}(x_l, x_r, z) \in L$ and $\text{cl}(p, q)$ is a set with at most four elements there is a $r_2 > 0$ so that for each $c \in \{a,b\}$ either:
    \begin{enumerate}
        \item $(\Phi_{\rho_c}(\gamma_{p q} p), \Phi_{\rho_c}(\gamma_{pq}q))\in \text{Far}_c(r_2)$,
        \item At least one of $\Phi_{\rho_c}(\gamma_{pq} p), \Phi_{\rho_c}(\gamma_{pq}q)$ is not in $U_{r_2, c}$.
    \end{enumerate}
    For $c \in \{a,b\}$ let $K_{r_2,c} \subset B_{n}^{2}$ be the compact set of pairs $(p, q)$ so that $\delta_0(\Phi_c(p),\Phi_c(q)) \geq r_2$.
    
    In the first case $(p, q) \in \mathscr{R}_{\kappa,\eta,\nu}$ so $\delta_0(\Phi_{\rho_c}(p), \Phi_{\rho_c}(q)) \leq \eta$ and $(\gamma_{pq} p, \gamma_{pq}q) \in K_{r_2,c}$.
    In the second case, there is a point $w \in \partial \Gamma$ and $p' \in \{p,q\}$ so $\delta_0( \Phi_{\rho_c}(p'), \Phi_{\rho_c}(w,w)) \leq \kappa$ and $(\gamma_{p q} p', \gamma_{pq}(w,w)) \in K_{r_2,c}$.
    We conclude from the Loxodromy Refinement Proposition that for any $\varepsilon' > 0$ there is a $\tau(\varepsilon') > 0$ so that for any $\kappa, \eta, \nu$ with $\kappa, \eta < \tau(\varepsilon')$ that $\rho_c(\gamma_{pq})$ is $(r,\varepsilon')$-biloxodromic for all $c \in \{a,b\}$, as desired.    \end{proof}

    We next use the control over quantitative loxodromy given by Claim \ref{claim-needed-proximality-improvement} to arrange for $p, q, \gamma_{pq}p,\gamma_{pq}q, \gamma_{pq}^+$, and $\gamma_{pq}^-$ to be sufficiently well-placed to give uniform control over the choice of affine charts in the main estimate.
    The following consequence of the hyperconvexity of $\xi_{\rho_a},\xi_{\rho_b}$ and that $\Phi_{\rho_c}$ is a homeomorphism of compact metric spaces is useful in this.

    \begin{claim}\label{claim-easy-homeo-compact}
        For any $\varepsilon_1 > 0$ there is a $\tau > 0$ so that for any $c \in \{a,b\}$, $k \in \mathcal{K}_n$, and $z \in \partial \Gamma$:
        \begin{enumerate}
            \item If $p = (x,y)_k$ and $\Phi_{\rho_c}(p) \in N_\tau(\xi^{n-1}_{\rho_c}(z))$ then $d_{\partial \Gamma}(x, z) < \varepsilon_1$ or $d_{\partial \Gamma}(y,z) < \varepsilon_1$,
            \item If $p = (x,y)_k$ and $\Phi_{\rho_c}(p) \in N_\tau(\xi^1_{\rho_c}(z))$ then $d_{\partial \Gamma}(x,z) < \varepsilon_1$ and $d_{\partial \Gamma}(y,z) < \varepsilon_1$.
        \end{enumerate}
    \end{claim}

    \begin{proof}
        We prove the first claim. The second is similar.
        The space ${B}_n \times \partial \Gamma$ is compact, and the function $\mu_n : B_n \times \partial \Gamma$ given by $(p, z) \mapsto \delta_0(\Phi_{\rho_c}(p), \xi^{n-1}_{\rho_c}(z))$ is continuous and zero exactly on the set $Z = \{((x,y)_k,z) \mid x = z \text{ or } y = z \}$, by the hyperconvexity of $\xi_{\rho_c}$.
        
        The set $K_{\varepsilon} \subset B_n \times \partial \Gamma$ given by $\{((x,y)_k, z) \mid d_{\partial \Gamma}(x,z) \geq \varepsilon_1 \text{ and } d_{\partial \Gamma}(y,z) \geq \varepsilon_1 \}$ is compact and disjoint from $Z$.
        So the restriction of $\mu_n$ to $K_{\varepsilon_1}$ has a positive infimum $\tau > 0$.
        So any $p \in B_n$ with $\Phi_{\rho_c}(p) \in N_\tau(\xi^{n-1}_{\rho_c}(z))$ has $(p,z) \in (B_n \times \partial \Gamma) - K_{\varepsilon_1}$. This proves the claim.
    \end{proof}

    We next constrain the placement of the edges of $\text{cl}(p,q)$ relative to $\gamma_{pq}^-$ and the placement of $\gamma_{pq} z$ relative to $\gamma_{pq}^+$. Recall $z = z(p, q)$ is the third point in $\partial \Gamma$ chosen in determining $\gamma_{pq}$.

    \begin{claim}\label{claim-basic-fp-separation}
        For all $\sigma^+,\sigma^- > 0$ there exist $\eta, \kappa, \nu, \varepsilon' > 0$ so that for any $(p, q) \in \mathscr{S}_{\eta, \kappa, \nu, \varepsilon'}$, the edges $x_l, x_r$ of ${\rm{cl}}(p, q)$ satisfy:
        \begin{enumerate}
            \item $d_{\partial \Gamma}(x_l,\gamma_{pq}^-), d_{\partial \Gamma}(x_r, \gamma_{pq}^-) < \sigma^-$,
            \item $d_{\partial \Gamma}(\gamma_{pq} z, \gamma_{pq}^+) < \sigma^+$.
        \end{enumerate}
    \end{claim}

    \begin{proof}
        We begin with the first claim.
        We use the same dichotomy as in the proof of Claim \ref{claim-needed-proximality-improvement}.
        First, let $(\Phi_{\rho_c}(p), \Phi_{\rho_c}(q)) \in \text{Far}_c(r_2)$, so one of $\Phi_{\rho_c}(\gamma_{p q} p)$ and $\Phi_{\rho_c}(\gamma_{p q}q)$ is not in $N_{r_2/2}(\xi^{1}_{\rho_c}(\gamma_{pq}^+))$.
        If $\gamma_{pq}$ is $(r,\varepsilon')$-biloxodromic with $\varepsilon' < r_2/2$, this implies that at least one of $\Phi_{\rho_c}(p)$ and $\Phi_{\rho_c}(q)$ is in $N_{\varepsilon'}(\xi_{\rho_c}^{n-1}(\gamma_{pq}^-))$.
        Without loss of generality, take this to hold for $p$.
        Claim \ref{claim-easy-homeo-compact} then shows that for $\varepsilon'$ sufficiently small, writing $p = (x,y)_k$, some $w \in \{x, y\}$ and $\gamma_{pq}^-$ have $d_{\partial \Gamma}(w, \gamma_{pq}^-) < \sigma^-/2 $.
        With $\nu$ chosen so $\nu < \sigma^-/2$, the claim holds.

        The other case is that for at least one $p' \in \{p,q\}$ there is a $c \in \{a,b\}$ so $\Phi_{\rho_c}(\gamma_{pq} p')$ is not in $U_{r_2,c}$.
        Then $\Phi_{\rho_c}(\gamma_{p q}p')$ is not in $N_{r_2}(\xi^1_{\rho_c}(\gamma_{pq}^+))$, and the claim follows similarly.

        We now arrange for the second conclusion to hold.
        By Claim \ref{claim-easy-homeo-compact}, there is a $\tau > 0$ so that if $\delta_0(\xi^1_{\rho_c}(\gamma_{pq}z), \xi^1_{\rho_c}(\gamma_{p q}^+)) < \tau$ for some $c \in \{a,b\}$, then $d_{\partial \Gamma}(\gamma_{p q} z, \gamma_{p q}^+) < \sigma^+$.
        Since we have arranged for $d_{\partial \Gamma}(z, x_l)$ and $d_{\partial \Gamma}(z,x_r)$ to be at least $\text{diam}(\partial \Gamma)/3$, there is a $r_4 > 0$ independent of $p, q$ so $\delta_0(\xi^1_{\rho_c}(z), \Phi_{\rho_c}(p')) \geq r_4$ for all $p' \in \{p, q\}$ and $c\in \{a,b\}$.
        Applying the first conclusion with possibly decreased $\sigma^-$ gives constants $(\kappa_0, \eta_0, \nu_0,\varepsilon'_0)$ so that for $(p, q) \in \mathscr{S}_{\kappa_0, \eta_0, \nu_0, \varepsilon'_0}$ the point $z$ is forced to be uniformly separated from $\gamma_{pq}^-$.
        Applying Claim \ref{claim-easy-homeo-compact} produces a $r_5 > 0$ so the separation $\delta_0(\Phi_{\rho_c}(z), \xi^{n-1}_{\rho_c}(\gamma_{pq}^-)) \geq r_5$ holds for all $(p,q) \in \mathscr{S}_{\kappa_0, \eta_0, \nu_0, \varepsilon'_0}$ and $c \in \{a,b\}$.
        Now, take $\varepsilon'$ so that $\varepsilon' < \min \{\tau, \varepsilon'_0, r_5 \}$.
        Then $\rho_c(\gamma_{p q})\xi^1_{\rho_c}(z) \in N_\tau(\xi^1_{\rho_c}(\gamma_{p q}^+))$ for all $(p,q) \in \mathscr{S}_{\kappa_0, \eta_0, \nu_0, \varepsilon'}$ and $c \in \{a,b\}$.
        So $d_{\partial \Gamma}(\gamma_{p q}z, \gamma_{p q}^+) < \sigma^+$ for all such $(p,q)$.
    \end{proof}

    The previous claim also allows us to separate both $p'$ and $\gamma_{pq}p'$ from $\gamma_{pq}^+$ for $p' \in \{p,q\}$.

    \begin{claim}\label{claim-upgraded-separation}
        There is a $r_6 > 0$ so that for all sufficiently small $(\kappa, \eta, \nu, \varepsilon') > 0$ and $(p,q) \in \mathscr{S}_{\kappa, \eta, \nu,\varepsilon'}$, for all $c \in \{a,b\}$ and $p' \in \{p,q\}$:
        \begin{enumerate}
            \item $\delta_0(\Phi_{\rho_c}(p'), \xi^{n-1}_{\rho_c}(\gamma_{p q}^+)) \geq r_6$,
            \item $\delta_0(\Phi_{\rho_c}(\gamma_{p q}p'), \xi^{n-1}_{\rho_c}(\gamma_{p q}^+)) \geq r_6$.
        \end{enumerate}
    \end{claim}
    \begin{proof}
        Corollary of Claims \ref{claim-easy-homeo-compact} and \ref{claim-basic-fp-separation} together with the compactness of $L$.
    \end{proof}

\medskip
    
\noindent (\ref{step-4-affine-charts}). We now define the affine charts in which we carry out our estimate.
    Let $(p, q) \in \mathscr{S}_{\kappa, \eta, \nu, \varepsilon'}$ be given for small parameters to be specified later, and let $\gamma_{p q}$ be associated.
    Let $x_f \in \{x_l,x_r\}$ be the edge point of $\text{cl}(p,q)$ for which $\gamma_{pq}x_f$ is as furthest from $\gamma_{pq}^-$.
    For $c \in \{a,b\}$, take an affine chart $\mathcal{A}_c$ for $\RP^{n-1}$ so that in $\mathcal{A}_c$:
        \begin{enumerate}
           \item $\xi^{1}_{\rho_c}(\gamma_{pq}^-)$ is the origin,
           \item $\xi^{n-1}_{\rho_c}(\gamma_{pq}^+)$ is the hyperplane at infinity,
           \item For a basis $v_1, v_2, ..., v_{n-1}$ of $\xi^{n-1}_{\rho_c}(\gamma_{pq}^+)$ of eigenvectors of the restriction of $\rho_c(\gamma_{pq})$ to $\xi^{n-1}_{\rho_c}(\gamma_{pq}^+)$ listed in order of decreasing modulus and with $v_n$ a nonzero vector in $\xi^1_{\rho_c}(\gamma_{pq}^-)$, the $x_{j}$ axis ($j = 1, ..., n-1$) is the intersection of $\text{span}\{v_n,v_j\}$ with $\mathcal{A}_c$,
           \item $\xi^{1}_{\rho_c}(\gamma_{p q} x_f) = (1, ..., 1)$.
        \end{enumerate}
    This specifies $\mathcal{A}_c$ completely.
    Note the final condition is possible because of hyperconvexity of $\xi_{\rho_c}$, in particular \cite{guichard2005dualite} Th\'eor\`eme 2 and Prop. 7.
    Let $\delta_{\mathcal{A}_c}$ be the Euclidean metric on $\mathcal{A}_c$.
    
    The following allows us to work in these well-adapted affine charts.
    
    \begin{claim}\label{claim-charts-good}
        There are $C, R < \infty$ and $\kappa,\eta, \nu, \varepsilon' > 0$ so that for all $(p, q) \in \mathscr{S}_{\kappa, \eta,\nu, \varepsilon'}$ for each $c \in \{a,b \}$ and $p' \in \{p,q\}$:
        \begin{enumerate}
            \item $\Phi_{\rho_c}(p') $ and $\Phi_{\rho_c}(\gamma_{pq}p')$ are in $N_{R}(0)$ in $\mathcal{A}_c$ (with respect to $\delta_{\mathcal{A}_c}$),
            \item On $N_R(0) \subset \mathcal{A}_c$ the metrics $\delta_{\mathcal{A}_c}$ and $\delta_0$ are $C$-multiplicatively comparable.
        \end{enumerate}
    \end{claim}

    \begin{proof}
        The key fact is that $\text{PGL}_n(\bbR)$ acts simply transitively on $(n+1)$-ples of points in $\RP^{n-1}$ in general position.
        Denote by $l_0, l_1, ..., l_n$ the points in $\RP^{n-1}$ so that in the model inducing the metric $\delta_0$ with $e_1, ..., e_n$ an orthonormal basis for $\bbR^n$ we have $l_i = [e_i]$ ($i = 1, ..., n$), and $l_0$ a fixed chosen point so $l_0, l_1, ..., l_n$ are in general position.
        Let $\mathcal{A}_0$ be the affine chart on $\RP^{n-1}$ specified by placing $l_n$ at the origin, $\text{span}\{l_1, ..., l_{n-1}\}$ as the hyperplane at infinity, the $x_j$-axis as $\text{span}\{v_n, v_j\}$ ($j = 1, ..., n-1$) and $l_0$ at $(1,..., 1)$.
        
        The affine chart $\mathcal{A}_c$ differs from $\mathcal{A}_0$ by the element $g_{pq}$ of $\text{PGL}_n(\bbR)$ that has $g[l_i] = v_i$ for $i = 1, ..., n$ and $g[l_0] = \xi^1_{\rho_c}(\gamma_{pq}x_f)$.
        To prove multiplicative comparability of $\delta_{\mathcal{A}_c}$ and $\delta_0$ on the complement of a neighborhood of fixed radius about $\xi^{n-1}_{\rho_c}(\gamma_{pq}^+)$, it thus suffices to place $\xi^1_{\rho_c}(\gamma_{pq}x_f), [v_1], ..., [v_n]$ inside of a fixed compact subset of the collection of $(n+1)$-ples of points in $\RP^{n-1}$ in general position.
    
        Let us first note that there is a $r_7 > 0$ independent of $p, q$ so that $d_{\partial \Gamma}(\gamma_{p q} x_f , \gamma_{pq}^-) > r_7$, $d_{\partial \Gamma}(\gamma_{pq}x_f, \gamma_{p q}^+) > r_7$, and $d_{\partial \Gamma}(\gamma_{pq}^-, \gamma_{pq}^+) > r_7$.
        Note the set $\mathscr{N}_{r_7}$ of $r_7$-separated triples of points in $\partial \Gamma$ is compact.
        With the convention $\xi^0(x) = \{0\}$ for all $x \in \partial \Gamma$ in the edge cases, $(\xi^{k-1}(x) \oplus \xi^{n-k}(z)) \cap \xi^1(y) = \{0 \}$ for any $(x,y,z) \in \partial \Gamma^{(3)+}$, $c \in \{a,b \}$, and $k \in \{1, ..., n\}$ by \cite{guichard2005dualite} Th\'eor\`eme 2 and Prop. 7.
        So there is a $r_8 > 0$ so that for all $(p, q) \in \mathscr{S}_{\kappa,\eta, \nu, \varepsilon'}$, with $\kappa,\eta,\nu\,\varepsilon'$ sufficiently small to apply Claim \ref{claim-upgraded-separation}, $c \in \{a,b\}$, and $(n-1)$-ples of eigenvectors $v_1, ...,v_{k-1}, v_{k+1}, ..., v_n$ of $\rho_c(\gamma_{pq})$, \begin{align*} \delta_0(\xi^1_{\rho_c}(\gamma_{pq} x_f), \text{span}(v_1, ..., v_{k-1}, v_{k+1}, ..., v_n))\geq r_8.
        \end{align*}

        This, together with Prop. \ref{prop-guichard-compact} places the point $(\xi^1_{\rho_c}(\gamma_{pq} x_f), [v_1], [v_2], ..., [v_n])$ inside of a fixed compact subset $\mathscr{K}$ of the space of $(n+1)$-ples of points in general position in $\RP^{n-1}$.
        So for every $r_{9} > 0$ there is a $C(r_9) < \infty$ so that $\delta_{\mathcal{A}_c}$ and $\delta_0$ are $C(r_9)$-comparable on the complement of a $r_{9}$-neighborhood of $\xi^{n-1}(\gamma_{pq}^+)$.

        Applying Claims \ref{claim-easy-homeo-compact} and \ref{claim-basic-fp-separation} produces for any sufficiently small $r_{9}> 0$ quality constants $\kappa,\eta, \nu, \varepsilon' > 0$ so that if $(p, q) \in \mathscr{S}_{\kappa,\eta, \nu,\varepsilon'}$ then $p, q , \Phi_{\rho_c}(p), \Phi_{\rho_c}(q)$ are not in $N_{r_{9}}(\xi_{\rho_c}^{n-1}(\gamma_{p q}^+))$ for all $c \in \{a,b\}$.
        The claim follows.
    \end{proof}

\medskip
\noindent (\ref{step-5-actual-computation}). The preliminaries to the main computation are now finished.
    In the affine chart $\mathcal{A}_c$,
    \begin{align*}
    \rho_c(\gamma_{pq}^{-1}) (x_1,x_2, ..., x_{n-1} ) &= \left( \frac{\lambda_n}{\lambda_{1}}x_1, \frac{\lambda_n}{\lambda_2}x_2, ..., \frac{\lambda_n}{\lambda_{n-1}}x_{n-1} \right) = \begin{pmatrix}
    \exp(-\ell^{\text{H}}_{\rho_c}(\gamma_{pq}^{-1}))x_1 \\
    \exp(-\ell^{\alpha_1+\alpha_2 + ... + \alpha_{n-2}}_{\rho_c}(\gamma_{pq}^{-1})) x_2 \\ 
    \vdots \\ 
    \exp(-\ell^{\alpha_1}_{\rho_c}(\gamma_{pq}^{-1}))x_{n-1} \end{pmatrix}^T \end{align*}
    so that for all $v, w \in \mathcal{A}_c$, \begin{align}
       e^{-\ell^{\text{H}}_{\rho_c}(\gamma_{pq}^{-1})}d_{\mathcal{A}_c}(v,w) \leq d_{\mathcal{A}_c}(\rho_c(\gamma_{pq}^{-1})v, \rho_c(\gamma_{pq}^{-1})w) \leq e^{-\ell^{\alpha_1}_{\rho_c}(\gamma_{pq}^{-1})} d_{\mathcal{A}_c}(v,w). \label{eq-slack-step}
    \end{align}
    
    Now let $(p, q)$ be $(\nu, \eta, \kappa, \varepsilon')$-adapted as in Claim \ref{claim-charts-good}.
    By Claim \ref{claim-cluster-characterization} applied to $L$ there is a $r_{10} > 0$ so that for all such $(p,q)$ and $c \in \{a,b\}$ the pair $(\Phi_{\rho_c}(\gamma_{pq}p) , \Phi_{\rho_c}(\gamma_{pq} q))$ is in $\Xi^2_c(r_{10})^2$ or $\text{Far}_c(r_{10})$.
    We then have
    \begin{align*}
            \frac{\delta_{\mathcal{A}_a}(\Phi_{\rho_a}(p), \Phi_{\rho_a}(q))}{\delta_{\mathcal{A}_b}(\Phi_{\rho_b}(p), \Phi_{\rho_b} (q))^{\alpha_{ab}^U}} &= \frac{\delta_{\mathcal{A}_a}(\rho_a(\gamma_{pq}^{-1}) \Phi_{\rho_a}(\gamma_{pq}p), \rho_a(\gamma_{pq}^{-1}) \Phi_{\rho_a}(\gamma_{pq}q))}
        {\delta_{\mathcal{A}_b}(\rho_b(\gamma_{pq}^{-1}) \Phi_{\rho_b} (\gamma_{pq}p), \rho_b(\gamma_{pq}^{-1})\Phi_{\rho_b}(\gamma_{pq} q))^{\alpha_{ab}^U}} \\
        &\leq \frac{e^{-\ell^{\alpha_1}_{\rho_a}(\gamma_{pq}^{-1})}}{e^{-\alpha_{ab}^U \ell^{\text{H}}_{\rho_b}(\gamma_{pq}^{-1})} } \frac{\delta_{\mathcal{A}_a}(\Phi_{\rho_a}(\gamma_{pq}p), \Phi_{\rho_a}(\gamma_{pq}q))}
        {\delta_{\mathcal{A}_b}(\Phi_{\rho_b} (\gamma_{pq}p), \Phi_{\rho_b}(\gamma_{pq} q))^{\alpha_{ab}^U}} \\
        &\leq \exp(\alpha_{ab}^U\ell^{\text{H}}_{\rho_b}(\gamma_{\rho_{pq}}^{-1}) -\ell^{\alpha_1}_{\rho_a}(\gamma_{pq}^{-1})) \max( C_{\Xi}(r_{10}), C_{\text{Far}}(r_{10}) ) \\
        &\leq \max( C_{\Xi}(r_{10}), C_{\text{Far}}(r_{10}) ).
        \end{align*}
    So $\Phi_{ab}$ is $\alpha_{ab}^U$-H\"older on a neighborhood $U$ of $\Xi^1_{\rho_b}$.
    Lemma \ref{lemma-holder-regularity-off-frenet} shows that $\Phi_{ab}$ is $\alpha_{ab}^U$-H\"older on the complement of a much smaller neighborhood $V$ of $\Xi^1_{\rho_b}$ with closure contained in $U$. The Lebesgue number lemma implies $\Phi_{ab}$ is $\alpha_{ab}^U$-H\"older.
\end{proof}

\subsection{Completeness}\label{ss-metrics-completeness}
The completeness properties of the bouquet coupling metric $D_B$ follow from a consequence of the Collar Lemma for Hitchin representations:

\begin{theorem}[Lee-Zhang \cite{leeZhang2017collar}, Corollary 1.4]\label{thm-collar-lemma-consequence} Let $\mathcal{C} = \{\gamma_1, ... \gamma_k\}$ be a collection of curves on $S$ containing a pants decomposition and so that the complement of $\mathcal{C}$ in $S$ is a union of disks. 
Then the map ${\rm{Hit}}_n(S) \to \bbR^k$ given by $\rho \mapsto (\ell_\rho^{\rm{H}}(\gamma_1), ..., \ell_\rho^{\rm{H}}(\gamma_k))$ is proper.
\end{theorem}

Theorems B and D follow from the identity $D_B^*(\rho_a, \rho_b) = \max(D_B(\rho_a, \rho_b), D_B(\rho_b, \rho_a))$ and:

\begin{proposition}
    Let $R > 0$ and $\rho_0 \in {\rm{Hit}}_n(S)$ for $n > 3$. Then the forward ball $B_{R,f}(\rho_0) =  \{ \rho\in {\rm{Hit}}_n(S) \mid D_B(\rho_0, \rho) \leq R \}$ is compact.
\end{proposition}

\begin{proof}
    Since $D_B$ induces the standard topology on $\text{Hit}_n(S)$, the forward ball $B_{R,f}$ is closed. So it suffices to find a compact set containing $B_{R,f}(\rho_0)$.
    Let $\mathcal{C} = \{\gamma_1, ..., \gamma_k\}$ be a set of curves satisfying the hypotheses of Theorem \ref{thm-collar-lemma-consequence}.
    If $\rho \in B_{R,f}(\rho_0)$ there is a $K$ so that $\ell^{\alpha_j}_{\rho}(\gamma) < K$ for all $\gamma_i \in \mathcal{C}$ and $j \in \{ 1, ..., n-1\}$ by Theorem C and finiteness of $\mathcal{C}$.
    So $\ell^{\text{H}}_{\rho}(\gamma) \leq (n-1)K$ for all $\gamma \in \mathcal{C}$ and $\rho \in B_{R,f}(\rho_0)$.
    Theorem \ref{thm-collar-lemma-consequence} shows that the collection of Hitchin representations $\rho$ with $\ell_\rho^{\text{H}}(\gamma) \leq (n-1)K$ for all $\gamma \in \mathcal{C}$ is compact.
\end{proof}

\bibliographystyle{plain}
\bibliography{refs}
\end{document}